\documentclass[11pt]{amsart}
\usepackage[margin=1.5in]{geometry}

\usepackage{mathrsfs}
\usepackage{amsfonts}
\usepackage{latexsym,epsfig}
\usepackage{mathabx}
\usepackage{amsmath, amscd, amsthm,amssymb}
\usepackage[pdftex]{color}
\usepackage{comment}
\usepackage{marginnote}
\usepackage[colorlinks,citecolor=blue,linkcolor=red]{hyperref}
\usepackage[nameinlink]{cleveref}

\newcommand{\RR}[0]{\mathbb{R}}
\newcommand{\HH}[0]{\mathbb{H}}

\newcommand{\Z}{\mathbb{Z}}

\newcommand{\A}[0]{\mathcal{A}}
\newcommand{\PP}[0]{\mathcal{P}}

\newcommand{\R}{\mathbb{R}}
\renewcommand{\sl}{\mathrm{slope}}

\newcommand{\sing}{\mathrm{sing}}

\newcommand{\diam}{\mathrm{diam}}
\newcommand{\s}{\mathbf{s}}
\renewcommand{\l}{\lambda}
\newcommand{\LL}{{\mathcal L}}
\newcommand{\CH}{{\operatorname{CH}}}

\usepackage{hyperref}

\newtheorem{theorem}{Theorem}[section]
\newtheorem{lemma}[theorem]{Lemma}
\newtheorem{proposition}[theorem]{Proposition}
\newtheorem{corollary}[theorem]{Corollary}

\newtheorem*{claim*}{Claim}

\theoremstyle{definition}
\newtheorem{remark}[theorem]{Remark}

\newcommand{\FF}{\mathcal{F}}

\newcommand{\cR}{\mathcal{R}}
\newcommand{\union}{\cup}
\newcommand{\intersect}{\cap}
\newcommand{\boundary}{\partial}

\newcommand\tsim{\kern-.4em\sim}
\newcommand\til{\widetilde}
\newcommand\ep{\epsilon}
\newcommand\cT{{\mathcal T}}
\newcommand\thull{{\mathbf t}}
\newcommand\rhull{{\mathbf r}}
\newcommand\ssm{\smallsetminus}

\renewcommand{\int}{\mathrm{int}}
\newcommand{\side}[1]{\mathcal{D}_{#1}}
\newcommand{\openside}[1]{\mathcal{D}_{#1}^{\mathrm{o}}}

\long\def\restate#1#2{
\medskip\par\noindent
{\bf \Cref{#1}} 
{\it #2}
\par\medskip
}

\long\def\realfig#1#2{
  \begin{figure}[htbp]
    \begin{center}
\graphicspath{ {./onelipschitz/} }
\includegraphics {#1}
\caption[#1]{#2}
\label{#1}
    \end{center}
\end{figure}}

\begin{document}
\title{Fibered faces, veering triangulations, and the arc complex}
\author[Y.N. Minsky]{Yair N. Minsky}
\address{Department of Mathematics\\ 
Yale University}
\email{\href{mailto:yair.minsky@yale.edu}{yair.minsky@yale.edu}}
\author[S.J. Taylor]{Samuel J. Taylor}
\address{Department of Mathematics\\ 
Temple University}
\email{\href{mailto:samuel.taylor@temple.edu}{samuel.taylor@temple.edu}}
\date{\today}
\thanks{This work was partially supported by NSF grants DMS-1311844
and DMS-1400498.}
\maketitle

\begin{abstract}
We study the connections between subsurface projections in curve and
arc complexes in fibered
3-manifolds and Agol's veering triangulation. The main theme is that
large-distance subsurfaces in fibers are associated to large simplicial
regions in the veering triangulation, and this correspondence holds
uniformly for all fibers in a given fibered face of the Thurston norm.
\end{abstract}

\section{Introduction}
\label{intro}

Let $M$ be a 3-manifold fibering over the circle with fiber $S$ and
pseudo-Anosov monodromy $f$. The stable/unstable laminations
$\lambda^+,\lambda^-$ of $f$ give rise to a function on the essential
subsurfaces of $S$,
$$
Y \mapsto d_Y(\lambda^+,\lambda^-), 
$$
where $d_Y$ denotes distance in the curve and arc complex of $Y$
between the lifts of $\lambda^\pm$ to the cover of $S$ homeomorphic to $Y$. This distance
function plays an important role in the geometry of the mapping class
group of $S$ \cite{MM2,BKMM, masur2013geometry}, and in the hyperbolic 
geometry of the manifold $M$ \cite{ECL1, ELC2}.

In this paper we study the function $d_Y$ when $M$ is fixed
and the fibration is varied. The fibrations of a given manifold are
organized by the faces of the unit ball of Thurston's norm on
$H_2(M,\boundary M)$, where each {\em fibered face} $\FF$ has the
property that every irreducible integral class in the open cone $\RR_+\FF$ 
represents a fiber. 
There is a pseudo-Anosov flow which is transverse to 
every fiber represented by $\FF$, and whose stable/unstable
laminations $\Lambda^\pm\subset M$ intersect each fiber to give the laminations
associated to its monodromy. With this we note that
the projection distance $d_Y(\lambda^+,\lambda^-)$ can be defined for any
subsurface $Y$ of any fiber in $\FF$.  We use $d_Y(\Lambda^+,\Lambda^-)$ to
denote this quantity. 

Our main results give explicit connections between $d_Y$ and
the {\em veering triangulation} of $M$, introduced by Agol \cite{agol2011ideal} and refined by
Gu\'eritaud \cite{gueritaud}, with the main feature being that when $d_Y$ satisfies
explicit lower bounds, a thickening of $Y$ is realized as
an embedded subcomplex of
the veering triangulation. In this way, the ``complexity'' of the
monodromy $f$ is visible directly in the triangulation in a way that
is independent of the choice of fiber in the face $\FF$.
This is in contrast with the results of \cite{ELC2} in which the estimates
relating $d_Y$ to the hyperbolic geometry of $M$ are
heavily dependent on the genus of the fiber.

The results are cleanest in the setting of a {\em fully-punctured}
fiber, that is when the singularities of the monodromy $f$ are assumed
to be punctures of the surface $S$ (one can obtain such examples by
starting with any $M$ and puncturing the singularities and their
flow orbits). All fibers in a face $\FF$ are fully-punctured when any one
is, and in this case we say that $\FF$ is a {\em fully-punctured face.}
In this setting $M$ is a cusped manifold and the veering triangulation
$\tau$ is an ideal triangulation of $M$.

We obtain bounds on $d_W(\Lambda^+,\Lambda^-)$ that hold for $W$ in any
fiber of a given fibered face: 

\begin{theorem}[Bounding projections over a fibered
    face] \label{th:bounding projections}
Let $M$ be a hyperbolic 3-manifold with fully-punctured fibered face
$\FF$ and veering triangulation $\tau$.
For any essential subsurface $W$ of any fiber of $\FF$,
\[
\alpha \cdot (d_W(\Lambda^- ,\Lambda^+) -\beta) < |\tau|, 
\]
where $|\tau|$ is the number of tetrahedra in $\tau$, 
$\alpha = 1$ and $\beta = 10$ when $W$ is an annulus and
$\alpha = 3|\chi(W)|$ and $\beta = 8$ when $W$ is not an annulus.
\end{theorem}

Note that
this means that $d_W \le |\tau| +10$ for each subsurface $W$,
no matter which fiber $W$ lies in.
Further, the complexity $|\chi(W)|$ of any subsurface $W$ of any fiber
of $\FF$ with $d_W(\Lambda^+,\Lambda^-) \ge 9$ is also bounded in 
terms of $M$ alone. 

In addition, given one fiber with a collection of subsurfaces
of large $d_Y$, we obtain control over the appearance of
high-distance subsurfaces in all other fibers:

\begin{theorem}[Subsurface dichotomy] \label{th:sub_dichotomy_fully_punctured}
Let $M$ be a hyperbolic 3-manifold with fully-punctured fibered face
$\FF$ and suppose that $S$ and $F$ are each fibers in $\R_+\FF$.
If $W$ is a subsurface of $F$, then either $W$ is isotopic
along the flow to a subsurface of $S$, or
$$3|\chi(S)| \ge d_W(\Lambda^-,\Lambda^+) -\beta,$$
where $\beta =10$ if $W$ is an annulus and $\beta = 8$ otherwise.
\end{theorem}

One can apply this theorem with $S$ taken to be the
smallest-complexity fiber in $\FF$. In this case there is some finite
list of ``large'' subsurfaces of $S$, and 
for all other fibers and all subsurfaces $W$ with $d_W$ sufficiently large,
$W$ is already accounted for on this finite list. 

For a sample
application of \Cref{th:sub_dichotomy_fully_punctured},
let $W$ be an essential annulus with core 
curve $w$ in a fiber $F$ of $M$ 
 and suppose that $d_W(\Lambda^-,\Lambda^+) \ge K$ for some $K > 10$. (We note
that it is easy to construct explicit examples  of $M$ with $d_W(\Lambda^-,\Lambda^+)$ 
as large as one wishes by starting with a pseudo-Anosov homeomorphism of $F$ 
with large twisting about the curve $w$.) If $w$ is trivial
in $H_1(M)$, then \Cref{th:sub_dichotomy_fully_punctured} 
(or more precisely \Cref{always subsurface}) implies that $w$ is actually isotopic
to a simple closed curve in \emph{every} fiber in the open cone $\RR_+\FF$ containing
$F$. When $w$ is nontrivial in $H_1(M)$ it determines a codimension-$1$ hyperplane
$P_w$ in $H^1(M) = H_2(M,\partial M)$ consisting of cohomology classes which 
vanish on $w$. For each fiber $S$ of $\RR_+\FF$ either $S$ is contained in $P_w$,
in which case $w$ is isotopic to a simple closed curve in $S$ as before, or $S$ lies 
outside of $P_w$ and $|\chi(S)| \ge \frac{K-10}{3}$.
We remark that the second alternative is non-vacuous so long as
$H^1(M)$ has rank at least 2.

The general (non-fully-punctured) setting is also approachable with
our techniques, but a number of complications arise and 
the connection to the veering triangulation of the fully-punctured manifold
 is much less explicit.
 An extension of the results in this paper to the 
general setting will be the subject of a subsequent paper.

\subsection*{Pockets in the veering triangulation}
When $Y$ is a subsurface of a fiber $X$ in $\FF$ and 
$d_Y(\Lambda^+,\Lambda^-)>1$, we show (\Cref{thm: tau-compatible}) 
that $Y$ is realized simplicially in the veering triangulation lifted
to the cover $X\times\R$. If $d_Y(\Lambda^+,\Lambda^-)$ is even larger then
this realization can be  thickened to a ``pocket'', which is a
simplicial region bounded by two isotopic copies of $Y$. With
sufficiently large assumptions this pocket can be made to embed in $M$
as well, and this is our main tool for connecting arc complexes to the
veering triangulation and establishing Theorems \ref{th:bounding
  projections} and \ref{th:sub_dichotomy_fully_punctured}:

\begin{theorem}\label{thm:pocket summary}
  Suppose $Y$ is a subsurface of a fiber $X$ with
  $d_Y(\lambda^-,\lambda^+) > \beta$, where $\beta=8$ if $Y$ is
  nonannular and $\beta=10$ if $Y$ is an annulus.
Then there is an embedded simplicial pocket $V$ in $M$ isotopic to a
thickening of $Y$, and
with $d_Y(V^-,V^+) \ge d_Y(\lambda^-,\lambda^+) - \beta$.
\end{theorem}
In this statement, $V^+$ and $V^-$ refer to the triangulations of the
top and bottom surfaces of the pockets, regarded as simplices in
the curve and arc complex $\A(Y)$. Also,  $d_Y(V^-,V^+)$ denotes
 the smallest $d_Y$-distance between an arc of $V^ -$ and an arc of $V^+$.\\

The veering triangulation in fact
recovers a number of aspects of the geometry of curve
and arc complexes in a fairly concrete way. As an illustration we prove

\begin{theorem}\label{th:total geodesic}
In the fully punctured setting, the arcs of the veering triangulation
form a geodesically connected subset $\A(\tau)$ of the curve and arc graph, in the sense that any two points in $\A(\tau)$ are connected by a geodesic that lies in $\A(\tau)$. 
\end{theorem}

\subsection*{Hierarchies of pockets}

One is naturally led to generalize \Cref{thm:pocket summary} from a
result embedding one pocket at a time to a description of all pockets
at once. Indeed \Cref{prop:disjoint_pockets} tells us that whenever
subsurfaces $Y$ and $Z$ of $X$ have large enough projection distances and are
not nested,
they have associated pockets $V_Y$ and $V_Z$ which are disjoint in 
$X \times \RR$.
These facts, taken together with
\Cref{th:total geodesic}, strongly suggest that the veering triangulation $\tau$ 
encodes the hierarchy of curve complex geodesics between $\lambda^\pm$ as 
introduced by Masur-Minsky in \cite{MM2}. We expect that, using a
version of \Cref{th:total geodesic} that applies to subsurfaces and
adapting the notion of ``tight geodesic'' from \cite{MM2}, one can
carry out a hierarchy-like construction within the veering triangulation
and recover much of the structure found in \cite{MM2}, with more
concrete control, at least in the 
fully-punctured setting. We plan to explore this approach in future
work. 

\subsection*{Related and motivating work}
The theme of using fibered 3-manifolds to study infinite families of
monodromy maps is deeply explored in McMullen \cite{mcmullen2000polynomial} and
Farb-Leininger-Margalit \cite{farb-leininger-margalit}, where the focus is on 
Teichm\"uller translation distance. 

Distance inequalities analogous to 
\Cref{th:sub_dichotomy_fully_punctured}, in the setting of Heegaard
splittings rather than surface bundles, appear in Hartshorn \cite{hartshorn}, 
and then more fully
in Scharlemann-Tomova \cite{scharlemann-tomova}. Bachman-Schleimer
\cite{BSc} use Heegaard surfaces to give bounds on the
curve-complex translation distance of the monodromy of a fibering. All
of these bounds apply to entire surfaces and not to subsurface
projections. In Johnson-Minsky-Moriah
\cite{johnson-minsky-moriah:subsurface}, subsurface projections are
considered in the setting of Heegaard splittings. A basic difficulty
in these papers which we do not encounter here is the compressibility of
the Heegaard surfaces, which makes it tricky to control essential
intersections. On the other hand,
unlike the surfaces and handlebodies that are used to obtain control
in the Heegaard setting, the foliations we consider here are infinite
objects, and the connection between them and finite arc systems in the
surface is a priori dependent on the fiber complexity. The veering
triangulation provides a finite object that captures this connection
in a more uniform way. 

The totally-geodesic statement of \Cref{th:total geodesic} should be compared to 
Theorem 1.2 of Tang-Webb \cite{tang-webb}, in which Teichm\"uller disks give rise to
quasi-convex sets in curve complexes. While the results of
Tang-Webb are more general, they are coarse, and it is interesting
that in our setting a tighter statement holds. Finally, we note that work by several authors
has focused on geometric aspects of the veering triangulation, including \cite{hodgson2011veering,futer2013explicit,hodgson2016non}.

\subsection*{Summary of the paper}
In \Cref{background} we set some notation and give Gu\'eritaud's
construction of the veering triangulation.
We also recall basic facts
about curve and arc complexes, subsurface projections and Thurston's
norm on homology. We spend some time in this section describing the
flat geometry of a punctured surface with an integrable
holomorphic quadratic
differential, and in particular giving an explicit description of the
circle at infinity of its universal cover (\Cref{same compactification}). 
While this is a fairly familiar picture, some delicate issues arise
because of the incompleteness of the metric at the punctures.

In \Cref{sections} we study {\em sections} of the veering
triangulations, which are simplicial surfaces isotopic to $X$ in the
cover $X\times\R$, and transverse to the suspension flow of the
monodromy. These can be thought of as triangulations of the surface
$X$ using only edges coming from the veering triangulation.
We prove \Cref{lem:extension} which says that a partial triangulation
of $X$ using only edges from $\tau$ can always be extended to a full section, and
\Cref{prop:connect} which says that any two extensions of a partial
triangulation are connected by a sequence of ``tetrahedron
moves''. This is what allows us to define and study the ``pockets''
that arise between any two sections.

In \Cref{hulls} we define a simple but useful construction, rectangle and
triangle hulls, which map saddle connections in our surface to unions
of edges of the veering triangulation. An immediate consequence of
the properties of these hulls is a proof of \Cref{th:total
  geodesic}.

In \Cref{surface_reps} we apply the flat geometry developed in
\Cref{background} to control the convex hulls of subsurfaces of the
fiber, and then use \Cref{hulls} to construct what we call
$\tau$-hulls, which are representatives of the homotopy class of a
subsurface that are simplicial
with respect to the veering triangulations. \Cref{thm: tau-compatible}
states that quite mild assumptions on $d_Y(\lambda^+,\lambda^-)$ imply
that the $\tau$-hull of $Y$ has embedded interior. The idea here is that any
pinching point of the $\tau$-hull is crossed by leaves of $\lambda^+$ and
$\lambda^-$ that intersect each other very little.
The main results of both \Cref{hulls} and \Cref{surface_reps} apply 
in a general setting and do not require that the surface $X$ be fully-punctured.

In \Cref{pockets} we put these ideas together to prove our main
theorems for fibered manifolds with a fully-punctured fibered face.
In \Cref{Y pocket} we describe the maximal
pocket associated to a subsurface $Y$ with $d_Y(\Lambda^+,\Lambda^-)$
sufficiently large (greater than 2, for nonannular $Y$).
We then introduce the notion of an {\em isolated pocket},
which is a subpocket of the maximal pocket 
that has good embedding properties in the manifold $M$. The existence
and embedding properties of these pockets are established in
\Cref{lem:iso_pocket} and \Cref{prop:disjoint_pockets},  which
together allow us to prove \Cref{thm:pocket summary}. 

From here, a simple counting argument gives \Cref{th:bounding projections}: the size of the
embedded isolated pockets is bounded from below in terms of
$d_Y(\Lambda^+,\Lambda^-)$ and $\chi(Y)$, and from above by the total number
of veering tetrahedra.

To obtain \Cref{th:sub_dichotomy_fully_punctured}, we use the pocket
embedding results to show that, if $Y$ is a subsurface of one fiber
$F$ and $Y$ essentially intersects another fiber $S$, then $S$ must
cross every level surface of the isolated pocket of $Y$, and hence the
complexity of $S$ gives an upper bound for $d_Y(\Lambda^+,\Lambda^-)$. To
complete the proof we need to show that, if $Y$ does not essentially
cross $S$, it must be isotopic to an embedded (and not merely immersed)
subsurface of $S$. This is handled by
\Cref{lem:embedding_fullly_punctured}, which may be of independent
interest. It gives a uniform upper bound for $d_Y(\Lambda^+,\Lambda^-)$ when
$Y$ corresponds to a finitely generated subgroup of $\pi_1(S)$, unless
$Y$ covers an embedded subsurface.

\subsection*{Acknowledgments}
The authors are grateful to Ian Agol and Fran\c{c}ois Gu\'eritaud for explaining their work to us. We also thank Tarik Aougab, Jeff Brock, and Dave Futer for helpful conversations and William Worden for pointing out some typos in an earlier draft. Finally, we thank the referee for a thorough reading of our paper and comments which improved its readability.

\newcommand{\Fr}{\operatorname{Fr}}
\section{Background}
\label{background}

The following notation will hold 
throughout the paper.
Let $\bar X$ be a closed Riemann surface
with an integrable
meromorphic quadratic differential $q$.
We remind the reader that $q$ may have poles of order $1$.
 We denote the vertical and
horizontal foliations of $q$ by $\lambda^+$ and $\lambda^-$
respectively.  Let $\PP$ be a finite subset of $\bar X$
that includes the poles of $q$ if any, and let
$X = \bar X \ssm \PP$.
Let $\sing(q)$ denote the union of $\PP$ with the set of zeros
of $q$.
We require further that $q$ has no horizontal or
vertical saddle connections, that is no leaves of $\lambda^\pm$ that
connect two points of $\sing(q)$.
This situation holds in particular if
$\lambda^\pm$ are the stable/unstable foliations of a pseudo-Anosov
map $f:X\to X$, which will often be the case for us.
If $\PP=\sing(q)$ (i.e. $\PP$ contains all zeros of $q$)
we say $X$ is {\em fully-punctured}.

Let $\hat X$ denote the metric completion of the universal cover $\til
X$ of $X$,  and note that there is an infinite branched covering
$\hat X \to \bar X$, infinitley branched over the points of $\PP$. The
preimage $\hat\PP$ of $\PP$ is the set of completion points.
The space $\hat X$ is a complete CAT$(0)$ space with the metric induced by
$q$.

\subsection{Veering triangulations}
\label{veering defs}
In this section let  $\PP=\sing(q)$. 
The veering triangulation, originally defined by Agol in
\cite{agol2011ideal} in the case 
where $q$ corresponds to
a pseudo-Anosov $f:X\to X$,  is an ideal
layered triangulation of $X\times\R$ which projects to a triangulation
of the mapping torus $M$ of $f$. The
definition we give here is due to Gu\'eritaud \cite{gueritaud}. 
(Agol's ``veering'' property itself will not actually play a role in
this paper, so we will not give its definition).

A {\em singularity-free rectangle} in $\hat X$ is an embedded rectangle whose
edges consist of leaf segments of the lifts of $\lambda^\pm$ and whose
interior contains no singularities of $\hat X$. 
If $R$ is a {\em maximal} singularity-free rectangle in  $\hat X$ then
it must contain a singularity on each edge. Note that there cannot be
more than one singularity on an edge since $\lambda^\pm$ have no
saddle connections.
We associate to $R$ an ideal tetrahedron whose vertices are $\boundary R
\intersect \hat\PP$, as in \Cref{gue-tetra}. 
This tetrahedron comes equipped with a ``flattening'' map into $\hat X$ as pictured.

\realfig{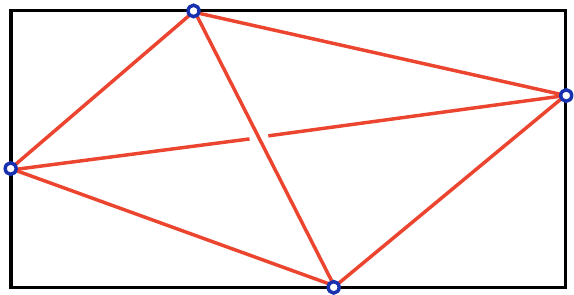}{A maximal singularity-free rectangle $R$ defines a
  tetrahedron equipped with a map into $R$.}

The tetrahedron comes with a natural orientation, inherited from the
orientation of $\hat X$ using the convention that the edge connecting
the horizontal boundaries of the rectangle lies {\em above} the edge
connecting the vertical boundaries. This orientation is indicated in \Cref{gue-tetra}.

The union of all these ideal tetrahedra, with faces identified whenever they
map to the same triangle in $\hat X$, is Gu\'eritaud's construction of
the veering triangulation of $\til X \times \R$.

\begin{theorem}\label{gueritaud construction} {\rm\cite{gueritaud}}
Suppose that $X$ is fully-punctured.
The complex of tetrahedra associated to maximal rectangles of $q$ is
an ideal triangulation $\til\tau$ of $\til X\times \R$, and the maps of
tetrahedra to their defining rectangles piece together to a fibration
$\pi:\til X \times \R \to \til X$. The action of $\pi_1(X)$ on $(\til
X,\til q)$ lifts simplicially
to  $\til\tau$, and equivariantly with respect to $\pi$.
The quotient is
a triangulation of $X \times \R$.

If $q$  
corresponds to a pseudo-Anosov $f:X\to X$
then the action of $f$ on $(X,q)$ lifts simplicially and
$\pi$-equivariantly to $\Phi:X\times\R\to X\times\R$.
The quotient is
a triangulation $\tau$ of the mapping torus $M$. The fibers of $\pi$
descend to flow lines for the suspension flow of $f$.
\end{theorem}

We will frequently abuse notation and use $\tau$ to refer to the
triangulation both in $M$ and in its covers. 

We note that a saddle connection $\sigma$ of $q$ is an edge of $\tau$
if and only if $\sigma$ spans a singularity-free rectangle in $X$. See
\Cref{extend-rect}.

\realfig{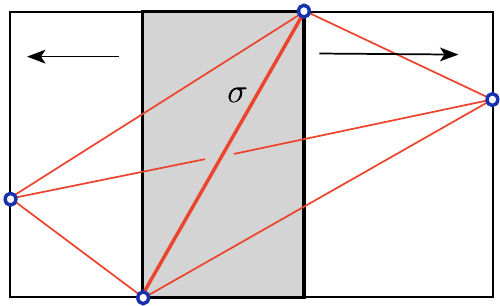}{The singularity-free rectangle spanned by $\sigma$ can be
  extended horizontally (or vertically) to a maximal one.}

If $e$ and $f$ are two crossing $\tau$-edges spanning rectangles $R_e$ and $R_f$,
note that $R_e$ crosses $R_f$ from top to bottom, or from left to
right -- any other configuration would contradict the singularity-free
property of the rectangles (\Cref{edges-cross}). If $\sl(e)$ denotes
the absolute value of the slope of $e$ with respect to $q$, we can see 
that $R_e$ crosses $R_f$ from top to bottom if and only if
$e$ crosses $f$ and 
$\sl(e) > \sl(f)$. We say that $e$ is {\em more vertical} than $f$ and
also 
write $e>f$. We will see that $e>f$ corresponds to $e$ lying higher
than $f$ in the uppward flow direction. 

Indeed we can see already that the relation $>$ is transitive, since
if $e>f$ and $f>g$ then the rectangle of $g$ is forced to intersect
the rectangle of $e$.

\realfig{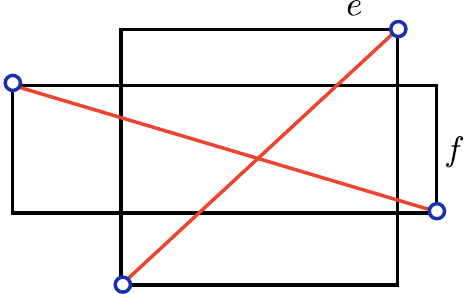}{The rectangle of $e$ crosses $f$ from top to
  bottom and we write $e>f$.}

We conclude with a brief description of  the local structure of $\tau$ around
an edge $e$: The rectangle spanned by $e$ can be extended
horizontally to define a tetrahedron lying below $e$ in the flow direction (\Cref{extend-rect}),
and vertically to define a tetrahedron lying above $e$ in the flow
direction. Call these $Q_-$ and $Q_+$ as in
\Cref{edge-swing}. Between these, on each side of $e$,
is a sequence of tetrahedra $Q_1,\ldots,Q_m$ $(m\ge 1)$ so that two successive
tetrahedra in the sequence $Q_-,Q_1,\ldots,Q_m,Q_+$ share a triangular
face adjacent to $e$. 
We find this sequence by starting with one
of the two top faces of $Q_-$, extending its spanning rectangle vertically until it hits
a singularity, and calling $Q_1$ the tetrahedron whose projection is inscribed in the new rectangle. If the new singularity belongs to $Q_+$ we are done $(m = 1)$, otherwise we repeat from the top face of $Q_1$ containing $e$ to find $Q_2$, and continue in this manner.
\Cref{edge-swing} illustrates this structure on one side of an edge
$e$. Repeating on the other side, note that the link of the edge $e$ is a circle, as expected.

\realfig{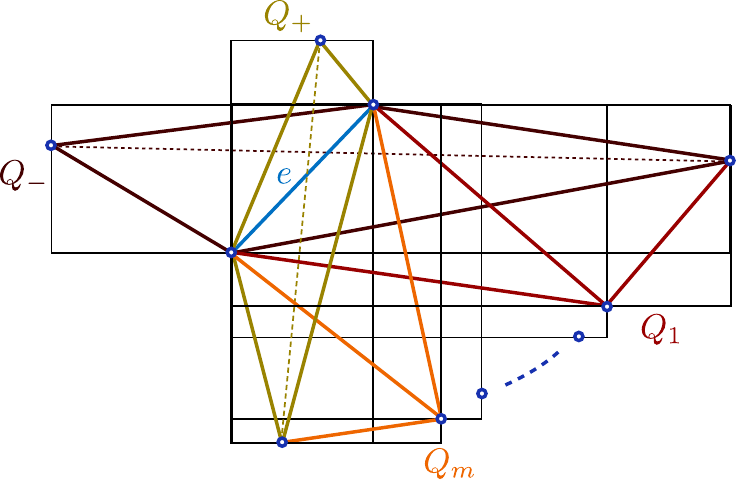}{The tetrahedra adjacent to an edge $e$ on one
  side form a sequence ``swinging'' around $e$}

\subsection{Arc and curve complexes}
\label{sec: arc_complex}

The arc and curve complex $\A(Y)$ for a compact surface $Y$ is usually
defined as follows:
its vertices are essential homotopy classes of 
embedded circles and properly embedded arcs $([0,1],\{0,1\}) \to
(Y,\boundary Y)$, where ``essential'' means not homotopic to a point or
into the boundary \cite{MM2}. We must be clear about the meaning of homotopy
classes here, for the case of arcs: If $Y$ is not an annulus,
homotopies of arcs are assumed to be homotopies of maps of pairs. When
$Y$ is an annulus the homotopies are also required to fix the endpoints.
Simplices of $\A(Y)$, in all cases, correspond to tuples of vertices
which can be simultaneously realized by maps that are disjoint on
their interiors. We endow $\A(Y)$ with the simplicial distance on its 
$1$-skeleton.

It will be useful, in the non-annular case, to observe that the
following definition is equivalent: Instead of maps of closed
intervals consider proper embeddings $\R \to \int(Y)$ into the
interior of $Y$, with equivalence arising from proper homotopy. This
definition is independent of the compactification of
$\int(Y)$. The natural isomorphism between these two versions of
$\A(Y)$ is induced by 
a straightening construction in a collar
neighborhood of the boundary.

If $Y\subset S$ is an essential subsurface  (meaning the inclusion of
$Y$ is $\pi_1$-injective and is not homotopic to a point or to an end
of $S$), we have subsurface projections $\pi_Y(\lambda)$ which are
defined for simplices $\lambda\subset \A(S)$ that intersect $Y$
essentially. Namely, after lifting $\lambda$ to the
cover $S_Y$ associated to $\pi_1(Y)$ (i.e. the cover to which $Y$ 
lifts homeomorphically and for which $S_Y \cong \int(Y)$), 
we obtain a collection of
properly embedded disjoint essential arcs and curves, 
which determine a simplex of $\A(Y)$.
We let $\pi_Y(\lambda)$ be the union of these vertices \cite{MM2}.
We make a similar definition for a lamination $\lambda$ that
intersects $Y$ essentially, except that we include not just the leaves
of $\lambda$ but all leaves that one can add in the complement of
$\lambda$ which accumulate on $\lambda$. This is natural when we
realize $\lambda$ as a measured {\em foliation} (as we do in most of
the paper), and need to include {\em generalized leaves}, which are
leaves that are allowed to pass through singularities. 
Note that the diameter of $\pi_Y(\lambda)$ in $\A(Y)$ is at most 2. 

Note that when  $Y$ is an annulus these arcs have natural endpoints
coming from the standard compactification of $\til S = \HH^2$ by a
circle at infinity. We remark that $\pi_Y$ does not depend on any choice
of hyperbolic metric on $S$.

When $Y$ is not an annulus and $\lambda$ and $\boundary Y$ are in
minimal position, we can also identify $\pi_Y(\lambda)$ with the
isotopy classes of components of $\lambda\intersect Y$.

These definitions naturally extend to 
immersed surfaces arising from covers of $S$. 
Let $\Gamma$ be a finitely generated subgroup of $\pi_1(S)$. Then the
corresponding cover $S_\Gamma \to S$ has a compact core $W$ --  a
compact subsurface $W \subset S_\Gamma$ such that $S_\Gamma \ssm
W$ is a collection of boundary parallel annuli. For curves or
laminations $\lambda^\pm$ of $S$, we have lifts
$\widetilde{\lambda}^\pm$ to $S_\Gamma$ and define
$d_W(\lambda^-,\lambda^+) = d_{S_\Gamma}(\widetilde{\lambda}^-,
\widetilde{\lambda}^+)$.

Throughout this paper, when $\lambda,\lambda'$ are two laminations or
arc/curve systems, we denote by $d_Y(\lambda,\lambda')$ the {\em
  minimal} distance between their images in $\A(Y)$,
 that is
$$
d_Y(\lambda,\lambda') = \min\{d_Y(l,l') : l \in \pi_Y(\lambda), l' \in  \pi_Y(\lambda') \}.
$$
To denote the \emph{maximal} distance between $\lambda$ and $\lambda'$ in $\A(Y)$ we write
$$
\diam_Y(\lambda,\lambda') = \diam_{\A(Y)}(\pi_Y(\lambda)\union\pi_Y(\lambda')). 
$$

\subsection{Flat geometry}
\label{AY in flat geometry}
In this section we return to the singular Euclidean geometry of $(X,q)$ 
and describe a circle at infinity for the flat metric
induced by $q$ on the universal cover $\til X$. 
We identify $\til X$ with $\HH^2$
after fixing a reference hyperbolic metric on $X$.
Because of incompleteness of the flat metric
at the punctures $\PP$, the connection between the circle we will describe
 and the usual circle at infinity for
$\HH^2$ requires a bit of care. A related 
discussion appears in Gu\'eritaud \cite{gueritaud}, although he deals
explicitly only with the fully-punctured case. 
With this picture of the
circle at infinity we will be able to describe $\pi_Y$ in terms of
$q$-geodesic representatives, and to describe a
$q$-convex hull for essential subsurfaces of $X$. In this section we do 
not assume that $X$ is fully-punctured.

The completion points $\hat \PP$ in $\hat X$ correspond
to parabolic fixed points for $\pi_1(X)$ in $\boundary \HH^2$, and
we abuse notation slightly by identifying $\hat\PP$ with this subset
of $\boundary \HH^2$.

A {\em complete $q$-geodesic ray} is either a geodesic ray
$r:[0,\infty)\to\hat X$ of infinite
length, or a finite-length geodesic segment that
terminates in $\hat\PP$.  A complete $q$-geodesic line is a geodesic
which is split by any point into two complete $q$-geodesic rays.
Our goal in this section is to describe a
circle at infinity that corresponds to endpoints of these rays.

\begin{proposition}\label{same compactification}
There is a compactification $\beta(\til X)$ of $\til X$ on which $\pi_1(X)$ acts
by homeomorphisms, with the following properties:
\begin{enumerate}
  \item There is a $\pi_1(X)$-equivariant homeomorphism $\beta(\til X) \to
    \overline{\HH^2}$, extending the identification of $\til X$ with $\HH^2$ and
taking $\hat \PP$ to the corresponding
    parabolic fixed points in $\boundary \HH^2$.
\item If $l$ is a complete $q$-geodesic line in $\hat X$ then its
  image in $\overline{\HH^2}$ is an embedded arc with endpoints on
  $\boundary\HH^2$ and interior points in $\HH^2 \cup \hat \PP$. Conversely, every pair of distinct points $x,y$ in
  $\boundary\beta(\til X) = \beta(\til X) \ssm \til X$ are the
  endpoints of a complete $q$-geodesic line.  The termination point in
  $\boundary\HH^2$ of a complete $q$-geodesic
  ray is in $\hat\PP$ if and only if it has finite length.
  \item The $q$-geodesic line connecting distinct
    $x,y\in\boundary\beta(\til X)$ is either unique, or there is a
    family of parallel geodesics making up an infinite Euclidean strip.
\end{enumerate}
\end{proposition}
One of the tricky points of this picture is that $q$-geodesic rays and
lines may meet points of the boundary $\boundary\beta(\til X)$ not just at their
endpoints. 

\begin{proof}
When $\PP = \emptyset$ and $X$ is a closed surface, $\til X$ is
quasi-isometric to $\HH^2$ and the proposition holds for the standard
Gromov compactification. We assume from now on that
$\PP\ne\emptyset$.  

We begin by  setting $\hat \HH^2 = \HH^2 \cup \hat \PP$ and endowing
it with the topology obtained 
by taking, for each $p \in \hat \PP$, horoballs based at $p$ as a
neighborhood basis for $p$. 

\begin{lemma}\label{hat same}
The natural identification of $\widetilde{X}$ with $\HH^2$
extends to a homeomorphism from $\hat X$ to $\hat \HH^2$.
\end{lemma}

\begin{proof}
First note that $\hat \PP$ is discrete as both a subspace of $\hat X$
and of $\hat \HH^2$.  Hence, it suffices to show that a sequence of
points $x_i$ in $\widetilde{X} = \HH^2$ converges to a point $p \in
\hat \PP$ in $\hat X$ if and only if it converges to $p$ in $\hat
\HH^2$. This follows from the fact that the horoball neighborhoods of
$p$ descend to cusp neighborhoods in $X$ which form a
neighborhood basis for the puncture that is equivalent to the
neighborhood basis of $q$-metric balls. 
\end{proof}

Our strategy now is to form the {\em Freudenthal space} of $\hat X$
and equivalently $\hat \HH^2$, which appends a space of {\em
  ends}. This space will be compact but not Hausdorff, and after a
mild quotient we will obtain the desired compactification which can be
identified with $\overline{\HH^2}$. Simple properties of this
construction will then allow us to obtain the geometric conclusions in
part (2) of the proposition.

Let $\epsilon(\hat X)$ be the space of ends of $\hat X$, that is the
inverse limit of the system of path components of complements of
compact sets in $\hat X$. The Freudenthal space $\Fr(\hat X)$ is the
union $\hat X\union \ep(\hat X)$ endowed with the toplogy generated by
using path components of complements of compacta to describe
neighborhood bases for the ends. Because $\hat X$ is not locally
compact, $\Fr(\hat X)$ is not guaranteed to be compact, and we have to take a bit of
care to describe it.

The construction can of course be repeated for $\hat\HH^2$, and 
the homeomorphism of \Cref{hat same} gives rise to a
homeomorphism $\Fr(\hat X) \to \Fr(\HH^2)$. 
Let us work in $\hat\HH^2$ now, where we can describe the ends
concretely using the following observations:

Every compact set
$K\subset \hat\HH^2$ meets $\hat\PP$ in a finite set $A$ (since $\hat
\PP$ is discrete in $\hat\HH^2$), and such a $K$ is
contained in an embedded closed disk $D$ which also meets $\hat\PP$ at
$A$. (This is not hard to see but does require attention to deal
correctly with the horoball neighborhood bases).
The components of $\hat\HH^2\ssm D$ determine a partition of
$\ep(\hat\HH^2)$, which in fact depends only on the set $A$ and not on
$D$ (if $D'$ is another disk meeting $\hat\PP$ at $A$, then $D\union
D'$ is contained in a third disk $D''$, and this  common refinement
of the neighborhoods gives the same partition).
Thus we have a more manageble (countable) inverse system of
neighborhoods in $\ep(\hat\HH^2)$, and with this description it is not
hard to see that $\ep(\hat\HH^2)$ is a Cantor set.

For each $p\in\hat\PP$ there are two distinguished ends $p^+,p^-\in
\ep(\hat\HH^2)$ defined as follows: For each finite subset $A\subset\hat\PP$
with at least two 
points one of which is $p$, the two partition terms adjacent to $p$ in the
circle (or equivalently, in the boundary of any $D\subset \hat \HH^2$
meeting $\hat\PP$ in $A$)
define neighborhoods in
$\ep(\hat\HH^2)$, and this pair of neighborhood systems determines $p^+$ and
$p^-$ respectively.

One can also see that $p^+$ (and $p^-$) and $p$ do not admit disjoint
neighborhoods, and this is why $\Fr(\hat\HH^2)$ is not Hausdorff. We
are therefore led to define the quotient space
\[
\beta(\hat \HH^2) = \Fr(\hat\HH^2)  / \sim , 
\]
where we make the identifications $p^- \sim p \sim p^+$, for each $p\in\hat\PP$.

We can make the same definitions in $\hat X$, obtaining
\[
\beta(\hat X) = \Fr(\hat X)  / \sim , 
\]
which we rename $\beta(\til X)$. 
Since the definitions are purely in terms of the topology of the spaces
$\hat\HH^2$ and $\hat X$, 
the homeomorphism of
 \Cref{hat same} extends to a homeomorphism $\beta(\til X) \to \beta(\hat\HH^2)$.

Part (1) of
\Cref{same compactification} follows once we establish that the
identity map of $\HH^2$ extends to a homeomorphism
$$
\beta(\hat\HH^2) \cong \overline{\HH^2}.
$$
This is not hard to see once we observe that the disks used above to
define neighborhood systems can be chosen to be ideal hyperbolic polygons. Their
halfspace complements serve as neighborhood systems for points of
$\boundary\HH^2\setminus\hat\PP$. A sequence converges in
$\overline{\HH^2}$ to  a point $p\in \hat\PP$ 
if it is eventually contained in any union of a horoball centered at p and two half-planes adjacent to $p$ on opposite sides.
This is modeled exactly by the equivalence relation $\sim$.

For part (2), let $D_0$ be a fundamental domain for $\pi_1(X)$ in
$\hat X$, which may be chosen to be a disk with vertices at points of
$\hat\PP$, and of finite $q$-diameter. Translates of $D_0$ can be glued
to build a sequence of nested disks $D_n$ exhausting $\hat X$, each of
which meets $\hat\PP$ in a finite set of vertices, and whose boundary is composed
of arcs of bounded diameter between successive vertices. 

A complete $q$-geodesic ray $r$ either has finite length and terminates in
a point of $\hat\PP$, or has infinite length in which case it leaves
every compact set of $\hat X$, and visits each point of $\hat\PP$ at
most once. Thus it must terminate in a point of
$\ep(\hat X)$ in the Freudenthal space. We claim that this point 
cannot be $p^+$ or $p^-$ for $p\in\hat\PP$. If $r$ terminates in
$p^+$, then for each disk $D_n$ ($n$ large) it must pass through the edge of
$\boundary D_n$ adjacent to $p$ on the side associated to $p^+$. Any
two such consecutive edges meet in $p$ at one of finitely many angles (images of
corners of $D_0$), and hence the accumulated angle between edges goes
to $\infty$ with $n$. If we replace these edges by their $q$-geodesic
representatives, the angles still go to $\infty$.
This means that $r$ contains infinitely many disjoint subsegments
whose endpoints are a bounded distance from $p$, but this contradicts the
assumption that $r$ is 
a geodesic ray.


The image of $r$ in the quotient $\beta(\til X)$ therefore terminates
in a point of $\hat\PP$ when it has finite length, and a point in
$\boundary\beta(\til X)\ssm \hat\PP$ otherwise.  The same is true for
both ends of a complete $q$-geodesic line $l$, and we note that both
ends of $l$ 
cannot land on the same point because then we would have a sequence of
segments $l_n\subset l$ of length going to $\infty$ with both
endpoints of $l_n$ on the same edge or on two consecutive edges
of $\boundary D_n$, a
contradiction to the fact that $l_n$ is a geodesic and the arcs in
$\boundary D_n$ have bounded $q$-length.

Now let $x,y$ be two distinct points in $\boundary\beta(\til
X)$. Assume first that both are not in $\hat\PP$. Then for large
enough $n$, they are in separate components of the complement of
$D_n$. If we let $x_i \to x$ and $y_i\to y$ be sequences in $\beta(\til X)$, then eventually $x_i$ and
$y_i$ are in the same components of the complement of $D_n$ as $x$ and
$y$, respectively. The
geodesic from $x_i$ to $y_i$ must therefore pass through the
corresponding boundary segments of $D_n$ and in particular through $D_n$,
so we can extract a
convergent subsequence as $i\to\infty$. Letting $n\to\infty$ and diagonalizing we
obtain a limiting geodesic which terminates in $x,y$ as desired. If
$x\in \hat\PP$ or $y\in\hat\PP$ the same argument works except that we
can take $x_i \equiv x$ or $y_i \equiv y$. This establishes part (2).

Now let $l$ and $l'$ be two $q$-geodesics terminating in $x$ and $y$.
If $x$ and $y$ are in $\hat\PP$ then $l=l'$ since the metric is
CAT(0). If $x\notin \hat\PP$ then both $l$ and $l'$ pass through
infinitely many segments of $\boundary D_n$ on their way
to $x$. Since these segments have uniformly bounded lengths, $l$ and
$l'$ remain a bounded distance apart. If $y\in\hat\PP$ then again
CAT(0) implies that $l=l'$, and if $y\notin\hat\PP$ then $l$ and $l'$
must cobound an infinite flat strip. This establishes part (3). 
\qedhere
\end{proof}

With \Cref{same compactification} in hand we can consider 
each complete $q$-geodesic line $l$ in $\hat X = \hat \HH^2$ as an arc
in the closed disk $\overline{\HH^2}$, which by the Jordan curve
theorem separates the disk ${\HH^2}$ into at least
$2$ components. Each component is an open disk whose closure meets
$\boundary\HH^2$ in a subarc of one of the complementary arcs of the endpoints of
$l$. We call the union of disks whose closures meet one of these complementary arcs
of the endpoints of $l$
an {\em open side} $\openside{l}$ of $l$. The closure of each open side in $\overline \HH^2$ is then a connected union
of closed disks, attached to each other along the points of $\hat\PP$
that $l$ meets on the circle. 
We call the closure of the open side $\openside{l}$ of $l$ in $\overline{\HH^2}$ the {\em side} $\side{l}$.  
Note that $\openside{l} = \int (\side{l} \cap \HH^2) = \side{l} \ssm (\partial \HH^2 \cup l)$, and if $\side{l}$ and $\side{l}'$ are the two sides of $l$, then $\side{l} \cap \side{l}' = l$.
See \Cref{line-disks}.

\realfig{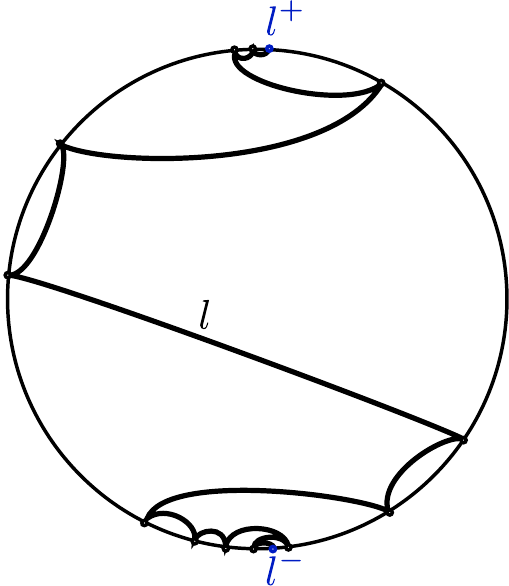}{A complete $q$-geodesic line $l$ ands its endpoints on $\partial \HH^2$.}

With this picture we can state the
following:

\begin{corollary} \label{cor:side_coherence}
Let $a,b$ be disjoint arcs in $\HH^2$ with well-defined, distinct endpoints on $\partial \HH^2$ and let $a_q,b_q$ be $q$-geodesic lines with the same endpoints as $a$ and $b$, respectively. Then $b_q$ is contained in a single side of $a_q$.
\end{corollary}

\realfig{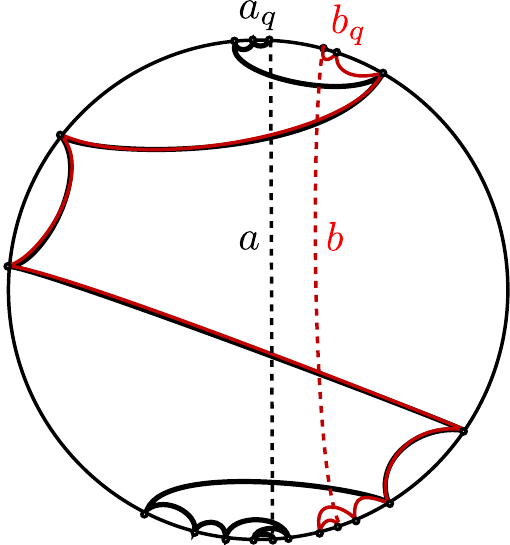}{Disjoint arcs with their $q$-geodesic representatives.}

\begin{proof}
Letting $L$ and $R$ be the arcs of $\boundary\HH^2$
minus the endpoints of $a$, 
the endpoints of $b$ must lie in one of them, say $L$, since
$a$ and $b$ are disjoint. 

Since $a_q$ and $b_q$ are geodesics in the CAT$(0)$ space $\hat X$,
their intersection is connected. If their intersection is empty, then
the corollary is clear. Otherwise, $b_q\ssm a_q$ is one or
two arcs, each with one endpoint on $a_q$ and the other on $L$. It follows
that $b_q\ssm a_q$ is on one open side of $a_q$, and the corollary
follows. 
\end{proof}

\subsection*{Subsurfaces and projections in the flat metric}      
Let $Y \subset X$ be an essential compact subsurface, and let
$X_Y=\til X/\pi_1(Y)$ be the associated cover of $X$. (Here we have identified $\pi_1(X)$ with the deck transformations of $\widetilde{X} \to X$ and fixed $\pi_1(Y)$ within its conjugacy class.)
For any
lamination $\lambda$ in $X$, we want to show that the projection
$\pi_Y(\lambda)$ can be 
realized by subsegments of the $q$-geodesic representative of $\lambda$.
Recall that $X$ is not necessarily 
fully-punctured.

We say a boundary component of $Y$ is {\em puncture-parallel} if it bounds a
disk in $\bar X \ssm Y$ that contains a single point of $\PP$. We
denote the corresponding subset of $\PP$ by $\PP_Y$ and refer to them
as the \emph{punctures} of $Y$. Let $\til{\PP}_Y$ denote the subset of
punctures of $X_Y$ which are encircled by the boundary components of
the lift of $Y$ to $X_Y$. 
In terms of the completed space $\bar X_Y$, $\til \PP_Y$ is exactly the set of completion points which have finite total angle.  
Let  $\boundary_0Y$ denote the union of the puncture-parallel components of
$\boundary Y$ and let $\boundary'Y$ denote the rest. Observe that the 
components of $\boundary_0 Y$ are in natural bijection with $\PP_Y$ and set
$Y' = Y\ssm\boundary_0Y$. 

Identifying $\til X$ with $\HH^2$, 
let $\Lambda\subset \boundary\HH^2$ be the limit set of
$\pi_1(Y)$, $\Omega = 
\boundary\HH^2 \ssm \Lambda$, and $\hat\PP_Y\subset \Lambda$
the set of parabolic fixed points of $\pi_1(Y)$.
Let $C(X_Y)$ denote the compactification of $X_Y$ given by
$(\HH^2 \union \Omega\union \hat\PP_Y)/\pi_1(Y)$, adding a point for each 
puncture-parallel end of $X_Y$, and a circle for each of the other
ends. Now given a lamination (or foliation) $\lambda$, realized geodesically in the
hyperbolic metric on $X$, its lift to $X_Y$ extends to properly
embedded arcs in $C(X_Y)$, of which the ones that
 are essential give
$\pi_Y(\lambda)$.

\Cref{same compactification}
allows us to
perform the same construction with the $q$-geodesic representative of
$\lambda$. Note that the leaves we obtain may 
meet points of $\til \PP_Y$
in their interior, but a slight perturbation
produces properly embedded lines in $X_Y$ which are properly isotopic
to the leaves coming from $\lambda$.

If $Y$ is an annulus the same construction works, with the observation
that the ends of $Y$ cannot be puncture-parallel and hence $C(Y)$ is
a closed annulus and the leaves have well-defined endpoints in its
boundary. We have proved:

\begin{lemma}
  \label{q arcs for AY}
Let $Y\subset X$ be an essential subsurface. 
If $\lambda$ is a proper arc or lamination in $X$ then the lifts of its $q$-geodesic
representatives to $X_Y$, after discarding inessential components, give representatives of $\pi_Y(\lambda)$.   
\end{lemma}

\subsection*{$q$-convex hulls}  
We will need a flat-geometry analogue of the hyperbolic convex hull.
The main idea is simple -- pull the boundary 
of the regular
convex hull tight using $q$-geodesics. The only difficulty comes from
the fact that
these geodesics can pass through parabolic fixed points, and
fail to be disjoint from each other, so 
the resulting object may fail to be an embedded surface. 
Our discussion is similar to Section $3$ of
Rafi \cite{rafi2005characterization}, but the discussion there
requires adjustments to handle correctly the incompleteness at punctures.

As above, identify $\til X$ with $\HH^2$. Let 
$\Lambda \subset \partial \HH^2$ be a closed set and let 
$\CH(\Lambda)$ be the convex hull of $\Lambda$ in $\HH^2$. We define
$\CH_q(\Lambda)$ as follows.

Assume first that $\Lambda$ has at least 3 points.
Each boundary geodesic $l$ of $\CH(\Lambda)$ has the same
endpoints as a
(biinfinite) $q$-geodesic $l_q$.  
By part (3) of \Cref{same compactification},
$l_q$ is
unique unless it is part of a parallel family of geodesics, making a
Euclidean strip.

The plane is divided by $l_q$ into two sides as in the discussion
before \Cref{cor:side_coherence},
and one of the sides, which we call  
$\side{l}$, meets $\partial \HH^2$
 in a subset of the complement of $\Lambda$.
Recall that $\side{l}$ is either a disk
or a string of disks attached along puncture
points.
If $l_q$ is one of 
a parallel family of geodesics, we include this family in $\side{l}$.
After deleting from $\hat X$ the interiors of $\side{l}$ for all $l$ 
in $\boundary \CH(\Lambda)$ (which are disjoint by \Cref{cor:side_coherence}),
we obtain $\CH_q(\Lambda)$, the
$q$-convex hull. 

If $\Lambda$ has 2 points then $\CH_q(\Lambda)$ is the closed Euclidean strip
formed by the union of $q$-geodesics joining those two points.

\medskip

Now fixing a subsurface $Y$ we can define a $q$-convex hull for the
cover $X_Y$, by taking a quotient of the $q$-convex hull of the limit
set $\Lambda_Y$ of $\pi_1(Y)$. This quotient, which we will denote  by
$\CH_q(X_Y)$, lies in the completion $\bar{X}_Y$.
Because $\CH_q(X_Y)$
may not be homeomorphic to $Y$, we pay explicit attention to a marking
map between $Y$ and its hull.

Let $\hat\iota:Y \to X_Y$ be the lift of the inclusion map to the
cover.

\begin{lemma}
  \label{q tight}
The lift $\hat\iota:Y\to X_Y$ is homotopic to a map $\hat\iota_q:Y\to \bar
X_Y$ whose image is the $q$-hull $\CH_q(X_Y)$ such that
\begin{enumerate}
\item  The homotopy $(h_t)_{t \in[0,1]}$ from $\hat\iota$ to $\hat\iota_q$ has the property that 
$h_t(Y) \subset X_Y$ for all $t \in [0,1)$.
\item Each 
component of $\partial_0 Y$
is taken by
    $\hat\iota_q$ to the corresponding completion point of $\til{\PP}_Y$.
\item If $Y$ is an annulus then the image of $\hat\iota_q$ is either a maximal flat cylinder in $\bar X_Y$ or the unique geodesic representative of the core of $Y$ in $\bar X_Y$.
\item If $Y$ is not an annulus then 
each component $\gamma$ of $\partial' Y$
 is taken by $\hat\iota_q$ to a
    $q$-geodesic representative in $\bar X_Y$. If there is a flat
    cylinder in the homotopy class of $\gamma$ then the interior of
    the cylinder is disjoint from $\hat\iota_q(Y)$.
  \item There is a deformation retraction $r:\bar X_Y \to
    \hat\iota_q(Y)$. For each component $\gamma$ of $\boundary'Y$, the
    preimage $r^{-1}(\hat\iota_q(\gamma))$ intersects $X_Y$ in either
    an open annulus or a union of open disks joined in a cycle along
    points in their closures.

\item If the interior $\int(\CH_q(\Lambda_Y))$ is a disk then
      $\hat\iota_q$ is a homeomorphism from $Y' = Y\ssm\boundary_0Y$  to its image.

\end{enumerate}
\end{lemma}

\begin{proof}
Let $\Gamma = \pi_1 Y$ and let
$\Lambda =\Lambda_Y \subset \partial \HH^2$ denote the limit set of $\Gamma$. 
As usual, $\CH(\Lambda)/\Gamma$ can be identified with $Y' = Y \ssm \partial_0 Y$. After
isotopy we may assume $\hat\iota:Y'\to \CH(\Lambda)/\Gamma$ is this
identification.

First assume that $Y$ is not an annulus.
Form $\CH_q(\Lambda)$ as above,
and for a boundary geodesic $l$ of
$\CH(\Lambda)$ define $l_q$ and its side $\side{l}$ as in the
discussion above. The quotient of $l_q$ is
a geodesic representative of a component of $\boundary Y$, and 
the quotient of the open side $\openside{l}$ in $X_Y$
is either an open annulus or a
union of open disks joined in a cycle along points in their completion.
The $q$-geodesic may pass
through points of $\hat \PP$,
so that there is a homotopy from $l$ to $l_q$ rel endpoints which
stays in $\HH^2$ until the last instant.

We may equivariantly deform the identity to a map
$\CH(\Lambda) \to \CH_q(\Lambda)$, which takes each $l$ to
$l_q$: since $\CH_q(\Lambda)$ is contractible, it suffices to give a
$\Gamma$-invariant triangulation of $\CH(\Lambda)$ and define the
homotopy successively on the skeleta. This homotopy descends to a map
from $Y'$ to $\CH_q(\Lambda)/\Gamma$, and can be chosen so that the
puncture-parallel boundary components map to the corresponding points of $\PP_Y$. This
gives the desired map $\hat\iota_q$ and establishes properties (1-4).

Using the description of the sides $\side{l}$, we may equivariantly
retract $\overline\HH^2$ to $\CH_q(\Lambda)$, giving rise to the
retraction $r$ of part (5). 

Finally, if the interior of $\CH_q(\Lambda)$ is a disk, then its
quotient is a surface. Our homotopy yields a
homotopy-equivalence of $Y'$ to this surface which preserves peripheral
structure and can therefore be deformed rel boundary to a
homeomorphism. We let $\hat\iota_q$ be this homeomorphism, giving part $(6)$.






When $Y$ is a (nonperipheral) annulus, $\Lambda_Y$ is a pair of points
and we recall from above that $\CH_q(\Lambda)$ is either a flat strip
in $\hat{X}$ which descends to a flat cylinder in $\bar X_Y$, or it is
a single geodesic. The proof in the annular case now proceeds exactly
as above. 
\end{proof}

Let $\iota_q : Y \to \bar X$ be the composition of $\hat \iota_q$ with
the (branched) covering $\bar X_Y \to \bar X$ and set 
$\partial_q Y = \iota_q(\boundary' Y)$. Note that this will be a 1-complex of saddle connections and not necessarily a homeomorphic image of $\boundary' Y$.

\subsection{Fibered faces of the Thurston norm}
A fibration $\sigma\colon M\to S^1$ of a finite-volume hyperbolic
3-manifold $M$ over the circle comes with the following structure:
there is an integral cohomology class in $H^1(M;\Z)$
represented by $\sigma_*:\pi_1M\to \Z$, which is the Poincar\'e dual
of the fiber $F$. There is a 
representation of $M$ as a quotient $F\times\R/\Phi$
where $\Phi(x,t) = (f(x),t-1)$ and  $f:F\to F$ is called the monodromy
map. This map is pseudo-Anosov and has stable and unstable (singular) measured foliations
$\lambda^+$ and $\lambda^-$ on $F$. Finally there is the suspension
flow inherited from the natural $\R$ action on $F\times\R$, and 
suspensions $\Lambda^\pm$ of $\lambda^\pm$ which are flow-invariant 2-dimensional
foliations of $M$. All these objects are defined up to isotopy.

The fibrations of $M$ are organized by the {\em Thurston norm}
$||\cdot||$ on $H^1(M;\R)$ \cite{thurston1986norm} (see also \cite{candel2000foliations}). This norm has a
polyhedral unit ball $B$ with the following properties:
\begin{enumerate}
  \item Every cohomology class dual to a fiber is in the
    cone $\R_+\FF$ over a top-dimensional open face $\FF$ of $B$.

  \item If $\R_+\FF$ contains a cohomology class dual to a fiber
      then {\em every} irreducible integral class in $\R_+\FF$ is dual to a
      fiber. $\FF$ is called a {\em fibered face} and its irreducible integral
      classes are called fibered classes.

  \item For a fibered class $\omega$ with associated fiber $F$, 
         $||\omega||=-\chi(F)$.
\end{enumerate}
In particular if $\dim H^1(M;\R)\ge 2$ and $M$ is fibered then there
are infinitely many fibrations, with fibers of arbitrarily large
complexity.
We will abuse terminology a bit by saying that a fiber (rather than
its Poincar\'e dual) is in $\R_+\FF$. 

The fibered faces also organize the suspension flows and the
stable/unstable foliations: If $\FF$ is a fibered face then there is a
single flow $\psi$ and a single pair $\Lambda^\pm$ of foliations whose leaves
are invariant by $\psi$, such that {\em every} fibration 
associated to $\R^+\FF$ may be isotoped so that its suspension flow is
$\psi$ up to a reparameterization, and the foliations $\lambda^\pm$ for the monodromy of its fiber $F$ are
$\Lambda^\pm\intersect F$.
These results were proven by Fried \cite{fried1982geometry}; see also
McMullen \cite{mcmullen2000polynomial}.

\subsection*{Veering triangulation of a fibered face}
A key fact for us is that the veering triangulation of the manifold
$M$ depends only on the fibered face $\FF$ and not on a particular
fiber. This was known to Agol for his original construction (see
sketch in \cite{agol-overflow}), but
Gu\'eritaud's construction makes it almost immediate.

\begin{proposition}[Invariance of $\tau$] \label{prop:invariance}
Let $M$ be a hyperbolic 3-manifold with fully-punctured fibered face
$\FF$.
Let $S_1$ and $S_2$ be
fibers of $M$ each contained in $\R_+ \FF$
and let $\tau_1$ and $\tau_2$ be the corresponding veering
triangulations of $M$. Then, after an isotopy preserving
transversality to the suspension flow,
$\tau_1 = \tau_2$. 
\end{proposition}
\begin{proof}
The suspension flow associated to $\FF$ lifts to the universal cover
$\til M$, and any fiber $S$ in $\R_+\FF$ is covered 
by a copy
of its universal cover $\til S$ in $\til M$ which meets every flow
line transversely, exactly once. Thus we may 
identify $\til S$ with the leaf space $\LL$ of this flow. The lifts
$\til\Lambda^\pm$ of the suspended laminations project to the leaf
space where they are identified 
with the lifts $\til\lambda^\pm$ of $\lambda^\pm$ to $\til S$.

The foliated rectangles used in the construction of $\tau$ from
 $\til{q}$ on $\til{S}$ depend 
 only on the (unmeasured) foliations $\til\lambda^\pm$.
Thus the abstract cell structure of $\tau$ depends
 only on the fibered face $\FF$ and not on the fiber. The map $\pi$ from each tetrahedron
 to its rectangle does depend a bit on the fiber, as we choose
 $q$-geodesics for the edges (and the metric $q$ depends on the
 fiber); but the edges are always mapped to arcs 
 in the rectangle that are transverse to both foliations. It follows
 that there is a transversality-preserving isotopy between the
 triangulations associated to any two fibers.
\end{proof}

\subsection*{Fibers and projections}
We next turn to a few lemmas relating subsurface projections over the various fibers in a fixed face of the Thurston norm ball. 

\begin{lemma}\label{lem:subgroup_projection}
If $\FF$ is a fibered face for $M$
and $Y \to S$ is an infinite covering where 
$S$ is a fiber in $\R_+\FF$ and $\pi_1(Y)$ is finitely generated,
then the projection distance $d_Y(\lambda^-,\lambda^+)$ depends only on $\FF$ and
the conjugacy class of the subgroup $\pi_1(Y) \le \pi_1(M)$ (and not
on $S$).
\end{lemma}

Note that $Y$ need not correspond to an embedded subsurface of $S$.

\begin{proof}
As in the proof of \Cref{prop:invariance}, $\til S$ can be
identified with the leaf space ${\LL}$ of the flow in $\til M$. 
The action of $\pi_1(M)$  on $\til M$ descends to $\LL$, and 
thus the cover $Y = \til S/\pi_1(Y)$ is identified with the quotient
${\LL}/\pi_1(Y)$ and the lifts of $\lambda^\pm$ to $Y$ are identified
with the images of $\til\Lambda^\pm$ in ${\LL}/\pi_1(Y)$. Thus the
projection $d_Y(\lambda^+,\lambda^-)$ can be obtained without
reference to the fiber $S$.
\end{proof}
This lemma justifies the notation $d_Y(\Lambda^+,\Lambda^-)$ used in
the introduction.

We will also require the following lemma, where we allow maps homotopic to fibers which are not necessarily embeddings.

\begin{lemma}\label{lem:flow_to_fiber_2}
Let  $F$ be a fiber of $M$. Let $Y\subset M$ be a
compact surface and let $h \colon F \to M$ be a map which is
homotopic to the inclusion. Suppose that $h(F) \cap Y$ is inessential
in $Y$, i.e. each component of the intersection is homotopic into the
ends of $Y$. Then the image of $\pi_1(Y)$ is
contained in $\pi_1(F) \vartriangleleft \pi_1(M)$.  
\end{lemma}

\begin{proof}
Let $\zeta$ be the  cohomology class dual to $F$. Since $h(F)$
meets $Y$ inessentially, every loop in $Y$ can be pushed off of $h(F)$
so $\zeta$ vanishes on $\pi_1(Y)$. But the
kernel of $\zeta$ in $\pi_1(M)$ is exactly $\pi_1(F)$, so the image of
$\pi_1(Y)$ is in $\pi_1(F)$. 
\end{proof}

\section{Sections and pockets of the veering triangulation}
\label{sections}

In this section the surface $X$ is fully-punctured.
A {\em section} of the veering triangulation $\tau$ is an 
embedding $(X,T) \to (X \times \R, \tau)$ which is simplicial with
respect to an ideal triangulation $T$ of $X$, and is a section of the
fibration $\pi \colon X \times \R \to X$ (hence transverse to the vertical flow). By \emph{simplicial} we 
mean that the map takes simplices to simplices.
The edges of $T$ are saddle connections of $q$ that are also edges of $\tau$ (i.e.
those which span singularity-free rectangles), and indeed any 
triangulation by $\tau$-edges gives rise to a section. 
We will abuse
terminology a bit by letting $T$ denote both the triangulation and the
section.

A {\em diagonal flip} $T\to T'$ between sections is an isotopy that
pushes $T$ through a single tetrahedron of $\tau$, either above it or below
it. Equivalently, if $R$ is a maximal rectangle and $Q$ its associated
tetrahedron, the bottom two faces of $Q$ might appear in $T$, in which
case $T'$ would be obtained by replacing these with the top two
faces. This is an upward flip, and the opposite is a downward flip.
We will refer to the transition as both a \emph{diagonal flip/exchange} and a \emph{tetrahedron move}, depending on the perspective.

An edge $e$ of $T$ can be flipped downward exactly when it is the
tallest edge, with respect to $q$, among the edges in either of the
two triangles adjacent to it. This makes $e$ the top edge of a
tetrahedron (i.e. the diagonal of a quadrilateral that connects the
horizontal sides of the corresponding rectangle). 
Similarly it can be flipped upward when it is the widest
edge among its neighbors.
See \Cref{flippability2}.

\realfig{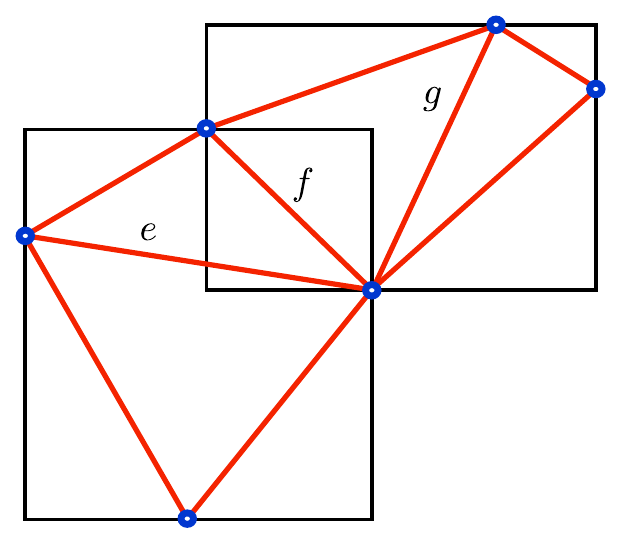}{The edge $e$ is upward flippable, $g$ is
  downward flippable, and $f$ is not flippable.}

In particular it follows that every section has to admit both an
upward and downward flip -- simply find the tallest edge and the
widest edge.

However it is not a priori obvious that a section even
exists. Gu\'eritaud gives an argument for this and more: 

\begin{lemma}[\cite{gueritaud}]\label{gue-sweep}
There is a sequence of sections $\cdots \to T_i\to
  T_{i+1}\to\cdots$ separated by upward diagonal flips, which sweeps through the
  entire manifold  $(X\times\R,\tau)$.
  Moreover, when $(X\times\R,\tau)$ covers the manifold $(M,\tau)$, this sequence is invariant by
  the deck translation $\Phi$.
\end{lemma}

We remark that Agol had previously proven a version of \Cref{gue-sweep} with his original definition of the veering triangulation \cite{agol2011ideal}.

For an alternative proof that sections exist, see the second proof of \Cref{lem:extension}.
We remark that \Cref{gue-sweep} does not give a complete picture of all possible
sections of $\tau$. In this section we will establish a bit more
structure. 

\medskip

For a subcomplex $K \le \tau$, denote by $T(K)$ the collection of
sections $T$ of $\tau$ containing the edges of $K$. A necessary
condition for $T(K)$ to be nonempty is that $\pi(K)$ is an
embedded complex in $X$ composed of $\tau$-simplices. We will continue
to blur the distinction between $K$ and $\pi(K)$. 

Our first result states that the necessary condition is sufficient:

\begin{lemma}[Extension lemma] \label{lem:extension}
Suppose that $E$ is a collection of
$\tau$-edges in $X$  with
pairwise disjoint interiors. Then $T(E)$ is nonempty.
\end{lemma}

The second states that $T(K)$ is always connected by tetrahedron
moves. This includes in particular the case of $T(\emptyset)$, the set
of all sections.

\begin{proposition}[Connectivity] \label{prop:connect}
  If $K$ is a collection of
$\tau$-edges in $X$  with
pairwise disjoint interiors,
  then $T(K)$ is connected
  via tetrahedron moves. 
\end{proposition}

\subsection*{Finding flippable edges}
Let $T$ be a section and let $\sigma$ be an edge of $\tau$, which is
not an edge of $T$. Any edge $e$ of $T$ crossing $\sigma$ must do so
from top to bottom ($e>\sigma$) or left to right ($e<\sigma$), as in
 \Cref{veering defs}, and we further note that all 
edges of $T$ that cross $\sigma$ do it consistently, all top-bottom or all
left-right, since they are disjoint from each other.

\begin{lemma} \label{lem:down_flip}
Let $T$ be a section and suppose that an edge $\sigma$ of $\tau$ is
crossed by an edge $e$ of $T$. If $e>\sigma$, then
there is an edge of $T$ crossing $\sigma$ which is downward flippable.
Similarly if $e<\sigma$ then
there is an edge of $T$ crossing $\sigma$ which is upward flippable.
\end{lemma}

\begin{proof}
Assuming the crossings of $\sigma$ are top to bottom,
let $e$ be the edge crossing $\sigma$ that has largest height with
respect to $q$. Let $D$ be a triangle of $T$ on either side of
$e$ and let $f$ be its tallest edge. 
Drawing the rectangle $M$ in which $D$ is inscribed (\Cref{tallest-crossing})
one sees that $R$, the rectangle of $\sigma$, is forced to cross it
from left to right. Hence, the edge $f$ must also cross $\sigma$.
Therefore, $f=e$ by choice of $e$.
It follows that $e$ is a downward flippable edge. 
\end{proof}

\realfig{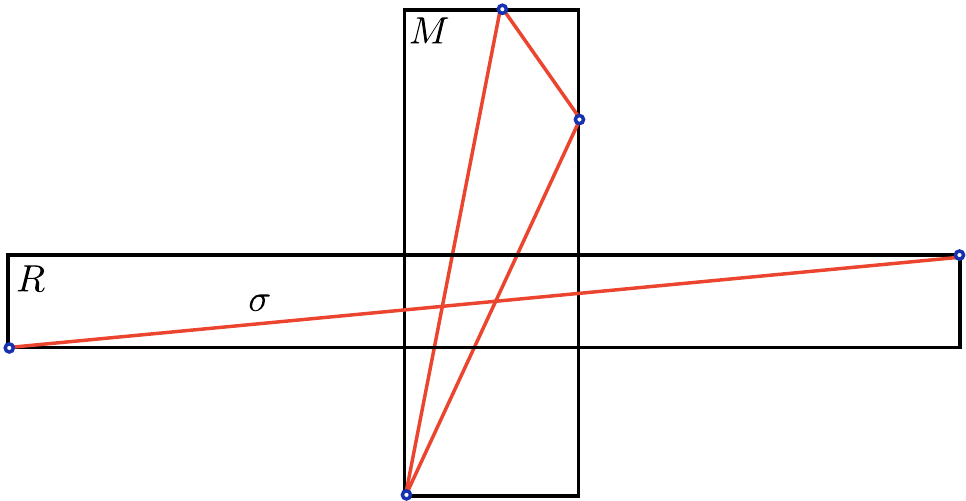}{The tallest $T$-edge crossing $\sigma$ must
  also be tallest in its own triangles.}

\subsection*{Pockets}
Let $T$ and $T'$ be
two sections and $K$ their intersection, as a subcomplex in $X\times\R$. Because both sections are embedded copies of $X$
transverse to the suspension flow, their union $T \cup T'$ divides
$X \times \R$ into two unbounded regions and some number of bounded
regions. Each bounded region $U$ is a union of tetrahedra bounded by two isotopic
subsurfaces of $T$ and $T'$, which correspond to a component $W$ of
the complement of $\pi(K)$ in $X$. The isotopy is obtained by
following the flow, and if it takes the subsurface of $T'$ upward to
the subsurface of $T$ we say that 
{\em $T$ lies above $T'$ in $U$}. We call $U$ a \emph{pocket over
  $W$}, and sometimes write $U_W$. We call $W$ the \emph{base} of the pocket $U$.

\begin{lemma}\label{lem:slope_drop}
 With notation as above, $T$ lies above $T'$ in the pocket $U_W$ if and only if,
 for every edge $e$ of $T$ in $W$ and edge $e'$ of $T'$ in $W$, if $e$ and $e'$ cross then $e>e'$.
\end{lemma}

Note that, for each edge $e$ of $T$ in $W$ there is in fact an edge $e'$ of
$T'$ in $W$ which crosses $e$, since both $T$ and $T'$ are
triangulations, with no common edges in $W$.

\begin{proof}
Suppose that $T$ lies above $T'$ in $U_W$ and let $e$ be an edge of
$T$ in $W$; hence, it is in the top boundary of 
$U$. Let $Q$ be the tetrahedron of $\tau$ for which $e$ is the top
edge. Via the local picture around $e$ (see \Cref{veering defs} and 
\Cref{edge-swing}), we see that $Q$ lies locally
below $T$. Its interior is of course disjoint from $T$ and $T'$ (and
the whole $2$- skeleton), hence it is inside $U$. Let $e_1$ be the
bottom edge of $Q$.
Note $e > e_1$. If $e_1$ is in $T'$,
stop (with $e' = e_1$). Otherwise it is in the interior of $U$, and we can repeat with
the tetrahedron for which $e_1$ is the top edge. We get a sequence of
steps terminating in some $e'$ in $T'$, which must be in the boundary
of $U$, and conclude $e > e'$ (by the transitivity of $>$ as in
\Cref{veering defs}). Now from the paragraph before \Cref{lem:down_flip},
the same slope relation holds for every edge of $T'$
crossing $e$, hence giving the first implication of the lemma. For the
other direction, exchange the roles of $T$ and $T'$ in the proof. 
\end{proof}

\subsection*{Connectedness of $T(K)$}
We can now prove \Cref{prop:connect}.
\begin{proof}
Let us consider $T$, $T'$ in $T(K)$. Let $U$ be one of the pockets,
and suppose $T$ lies above $T'$ in $U$. \Cref{lem:slope_drop} together
with \Cref{lem:down_flip} implies that $T$ has a downward flippable
edge $e$ which crosses an edge of $T'$ that is in $W$. In particular
$e$ itself is in $W$.
Performing this flip we reduce the number of tetrahedra contained in pockets. Thus a finite number of moves will take $T$ to $T'$, without disturbing $K$.
\end{proof}

As a consequence of \Cref{prop:connect} and its proof we have: 
\begin{corollary}\label{cor:top}
If $K$ is a nonempty subcomplex of $\tau$ and $T(K) \neq \emptyset$,
then there are unique sections $T^+(K)$ and  $T^-(K)$ in  $T(K)$ such
that every $T \in T(K)$ can be upward flipped to $T^+(K)$ and downward flipped
to $T^-(K)$. 
\end{corollary}

\begin{proof}
First note that $T(K)$ is finite: because $\tau$ is locally finite at the edges, there are only finitely many choices for a triangle adjacent to $K$. We then enlarge $K$ successively, noting that there is a bound on the number of triangles in a section.
Thus there exists a section $T^+$ in $T(K)$ which is not upward
flippable in $T(K)$.
For any two sections $T_1,T_2\in T(K)$ there is a $T_3\in
T(K)$ obtained as the union of the tops of the pockets of $T_1$ and
$T_2$ and their intersection. Thus $T_1$ is upward flippable unless
$T_1=T_3$, and similarly for $T_2$. This implies that $T^+$ is the
unique section in $T(K)$ which is not upward flippable, and every
other section is upward flippable to $T^+$. 
We define $T^-$ analogously.
\end{proof}

The section $T^+(K)$ is called the \emph{top of $T(K)$} and the
section $T^-(K)$ is called the \emph{bottom of $T(K)$}. Note that any
section obtained from $T^+(K)$ by upward diagonal exchanges is not in
$T(K)$.

\subsection*{Extension lemma}
We conclude this section with two proofs of \Cref{lem:extension}.

\begin{proof}[Proof one]
  \Cref{gue-sweep} gives us, in particular, the existence of at
  least one section $T_0$ which is disjoint from $E$, which we may
  assume lies above every edge of $E$.  

Then by \Cref{lem:down_flip} there is a downward
  flippable edge $e$ in $T_0$. The tetrahedron involved in the move lies
  above $E$, so $E$ still lies below (or is contained in) the new
  section $T_1$. We repeat this process, and at each stage every edge
  of $E$ is either contained in $T_i$ or crosses an edge of $T_i$ and
  lies below it. Thus by \Cref{lem:down_flip}, unless $E\subset T_i$ each $T_i$
  contains a downward flippable edge that is not contained in
  $E$. 

  Because $\tau$ is locally finite at each
  edge, {\em any} sequence of downward flips is a proper sweepout of
  the region below $T_0$, and hence
  must eventually meet every edge of $\tau$ below $T_0$. Thus we may
  continue until every edge of $E$ lies in $T_i$. 
\end{proof}

\begin{proof}[Proof two] Our second proof does not use \Cref{gue-sweep}, and in
  particular it gives an independent proof of the existence of
  sections.

Let $D$ be a component of the complement of $E$ which is not a triangle.
Let $e$ be an edge of $\boundary D$ and consider the collection of
$\tau$-tetrahedra adjacent to $e$. These contain a sequence
$Q_-,Q_1,\ldots Q_m,Q_+$, as in \Cref{edge-swing}, where $Q_-$ is the
tetrahedron with $e$ as its top 
edge, $Q_+$ is the tetrahedron with $e$ as its bottom edge, and the
rest are adjacent to $e$ 
on the same side as $D$ (if $D$ meets $e$ on two sides we just choose one). 
Two successive tetrahedra in this sequence
share a triangular face. We claim that one of
these faces must be contained in $D$. Equivalently we claim that one
of the triangles is not crossed by any edge of $E$.

Since each tetrahedron $Q$ is inscribed in a singularity free rectangle $R$, if an edge
$f$ of $E$ crosses any edge of $Q$ its rectangle crosses
all of $R$. It follows immediately, since the edges of $E$
have disjoint interiors, that they consistently cross $R$ all vertically, or
all horizontally. Because successive tetrahedra in the sequence share
a face it follows inductively that, if all the faces are crossed
by $E$, then they are all consistently crossed horizontally,
or all vertically.

However, $Q_-$ can 
only be crossed vertically by $E$ (since $E$ does
not cross $e$). Similarly $Q_+$  can only be crossed horizontally. It
follows that there must be a triangular face $F$ that is {\em not} crossed
by $E$. Thus $F$ is contained in $D$. Since $D$ is not a triangle,
at least one edge of $F$ passes through the
interior of $D$. We add this edge to $E$ and
proceed inductively. 
\end{proof}

\section{Rectangle and triangle hulls}
\label{hulls}

In this section we discuss a number of constructions that associate a 
configuration of $\tau$-edges to
a saddle connection of the quadratic differential $q$. These
will be used later to show that subsurfaces with large projection are 
compatible with the veering triangulation in the appropriate sense.
As a byproduct of our investigation, we prove the (to us) unexpected result
(\Cref{th:total geodesic})
that the edges of the veering triangulation form a totally 
geodesic subgraph of the curve and arc graph of $X$.

We emphasize that in \Cref{sec:rectangles_along_saddles} and 
\Cref{sec:t_hulls}, the surface $X$ is not necessarily
fully-punctured. Thus by $\tau$ we mean the veering triangulation
associated to the fully-punctured surface $X \ssm \sing(q)$. We will
say that a saddle connection of $X$ is a {\em $\tau$-edge} if its interior
is an edge of this veering triangulation. In particular this means
that its lift to $\til X$ spans a singularity-free rectangle.

\subsection{Maximal rectangles along a saddle connection}
\label{sec:rectangles_along_saddles}

Let $\sigma$ be a saddle connection, for the moment in the completed universal
cover $\hat X$. Consider the set $\cR(\sigma)$ of all rectangles which are {\em maximal 
with respect to the property that $\sigma$ passes through a diagonal}. 
Thus each $R\in\cR(\sigma)$ contains singularities in at least two
edges. Let $h(R)$ be the convex hull in $R$ of the singularities
in the boundary of $R$ and let $h^{(1)}(R)$ denote its $1$-skeleton
(see \Cref{rect-hulls}).

\realfig{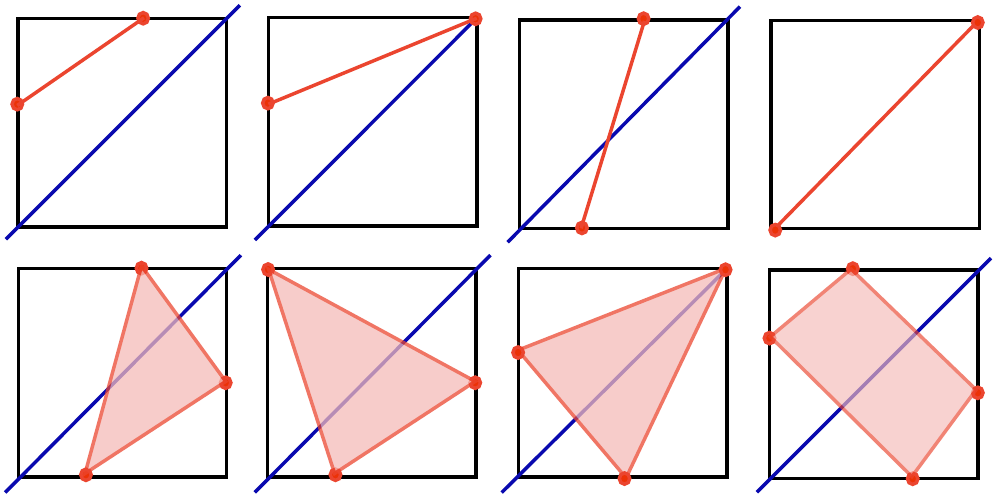}{The eight possible (up to symmetry) convex hulls $h(R)$, assuming
  at most one singularity per leaf of $\lambda^\pm$. The saddle
  connection $\sigma$ is in blue.}

Let
$$
\rhull(\sigma) = \bigcup \{  h^{(1)}(R): R\in \cR(\sigma) \}.
$$
See \Cref{r-example} for an example.
Note that all the saddle connections in $\rhull(\sigma)$ are edges of $\tau$ --- 
each of these arcs spans a singularity-free rectangle by construction. 
Moreover,  $\rhull(\sigma) = \{\sigma\}$ if $\sigma$ is itself a $\tau$-edge.  

\realfig{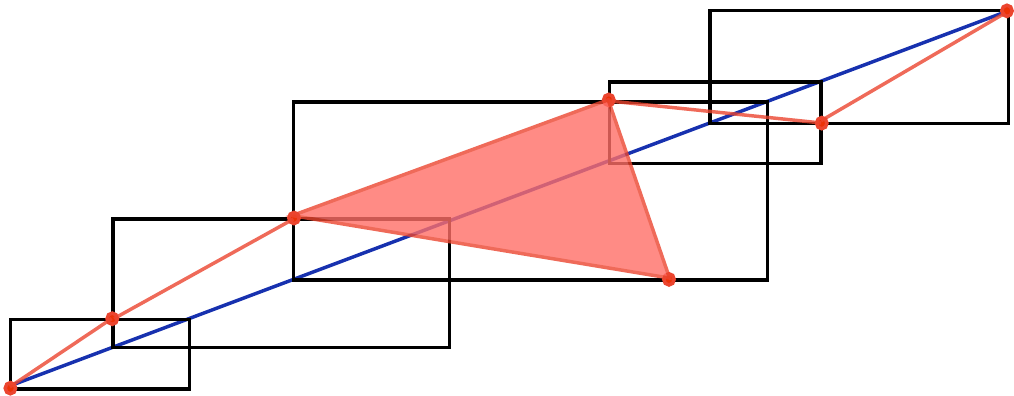}{Example of $\rhull(\sigma)$ (in red)}

The following lemma will play an important role throughout 
this paper.

\begin{lemma}\label{r disjoint}
  If saddle connections $\sigma_1$ and $\sigma_2$ have no transversal intersections then
  neither do $\rhull(\sigma_1)$ and $\rhull(\sigma_2)$. 
\end{lemma}

\begin{proof}  
Say that two rectangles meet {\em crosswise} if their interiors
intersect, and no corners of one are in the interior of the
other. Note that when two distinct rectangles meet crosswise,
any two of their diagonals intersect. We say that the rectangles meet
{\em properly crosswise} if they also do not share any corners, in
which case any two diagonals intersect in the interior.

Let $\tau_1$ and $\tau_2$ be saddle connections in $\rhull(\sigma_1)$
and $\rhull(\sigma_2)$, respectively, and suppose that they intersect
transversely. Hence their spanning rectangles $Q_1$ and $Q_2$ must
cross as in \Cref{edges-cross}. Assume that $Q_1$ is the taller and
$Q_2$ the wider.

Now let $R_1$ and $R_2$ be the rectangles of $\cR(\sigma_1)$ and
$\cR(\sigma_2)$ containing $Q_1$ and $Q_2$, respectively. 
Because of the singularities in the corners of $Q_1$ and $Q_2$, $R_2$
is no taller than $Q_1$ and $R_1$ is no wider than $Q_2$. Hence $R_1$
and $R_2$ meet crosswise. (See \Cref{crossing-rects}).

\realfig{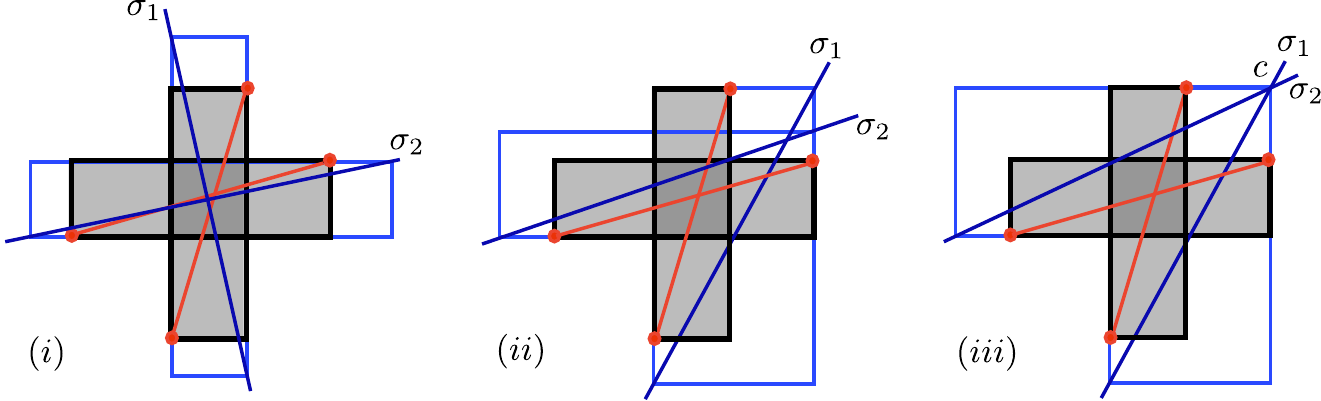}{Three examples of the crossing pattern. The rectangles $R_1$
  and $R_2$ are in blue, $\tau_1$ and $\tau_2$ are in red, and  $Q_1$ and $Q_2$ are shaded. In {\em (i)} and
  {\em (ii)} the
  crossing is proper. In {\em (iii)} the corner $c$ is shared.}

If they met properly crosswise then $\sigma_1$ and $\sigma_2$ would
have an interior intersection, which is a contradiction. Hence $R_1$
and $R_2$ share a corner $c$. But the edges meeting at $c$ would
have to pass through boundary edges of $Q_1$ and $Q_2$. Those edges already
have the singularities of $\tau_1$ and $\tau_2$, and so $c$ cannot be
a singularity. Thus if $c$ is the intersection of the diagonals
contained in $\sigma_1$ and $\sigma_2$ it would be in the interior of
both saddle connections, again a contradiction.

We conclude that $\tau_1$ and $\tau_2$ cannot cross.
\end{proof}

An immediate consequence of \Cref{r disjoint} is that we can carry out the
construction downstairs: If $\sigma$ is a saddle connection in $\bar X$ we
can construct $\rhull(\hat \sigma)$ for each of its lifts $\hat \sigma$ 
to $\hat X$, and the lemma tells us none of them intersect
transversally. Thus the construction projects downstairs to give a
collection of $\tau$-edges with disjoint interior.
Moreover if $K$ is {\em any} collection
of saddle connections with disjoint interiors then $\rhull(K)$ makes sense as a subcomplex of 
$\tau$ supported on some section by \Cref{lem:extension}. 
Hence, we will continue to use $\rhull(\cdot)$ to denote 
the corresponding map on saddle connections of $\bar X$. We remark that although
$\rhull(\cdot)$ takes collections of saddle connections with disjoint interiors to collections of
$\tau$-edges with disjoint interiors, it may do so with multiplicity. 

\subsection{Triangle hulls} \label{sec:t_hulls}

Now let us consider a similar operation that uses right triangles
instead of rectangles, and associates to a transversely oriented
saddle connection in the universal cover a homotopic path of saddle
connections.

If $\sigma$ is a saddle connection in $\hat X$ equipped with a
transverse orientation, let $\cT(\sigma)$ denote the collection of
Euclidean right triangles which are {\em maximal with respect to the
property that they are attached along the hypotenuse to $\sigma$ along
the side given by its transverse orientation}. A triangle $t$ in
$\cT(\sigma)$ must have exactly one singularity in each of its legs,
and so their convex hull $h(t)$ is a single saddle connection.
The set $\cT(\sigma)$ must be finite, and its hypotenuses cover
 $\sigma$ in a sequence
of non-nested intervals, 
ordered by their left (or right) endpoints. See \Cref{t-example}. 
Let $\thull(\sigma)$ be the union of segments $h(t)$ for $t\in\cT(\sigma)$.

\realfig{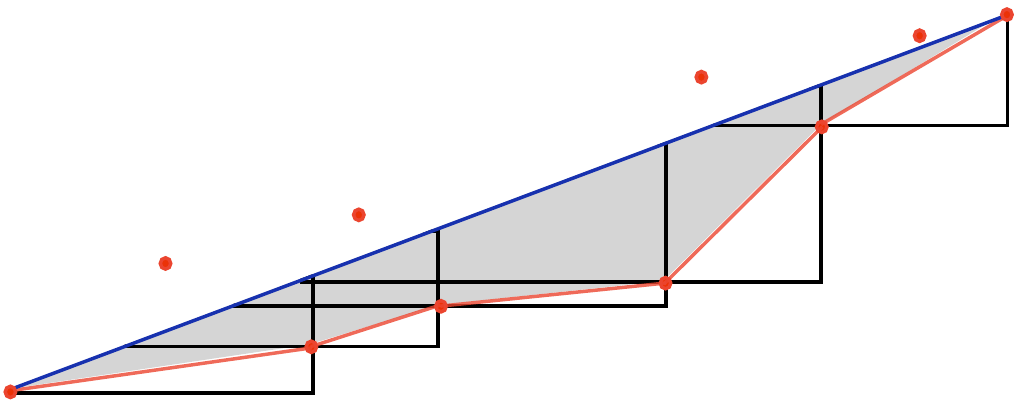}{An example of $\thull(\sigma)$ and $P(\sigma)$.}

\begin{lemma}\label{thull structure}
Either $\thull(\sigma) = \sigma$ or
$\sigma\union\thull(\sigma)$ is the boundary of an embedded Euclidean
  polygon $P(\sigma)$ in $\hat X$ which is foliated by arcs of $\lambda^\pm$.
\end{lemma}

\begin{proof}
  Suppose that $t$ and $t'$ are triangles of $\cT(\sigma)$ and $p\in
  t\intersect t'$ is in the interior of $t$. Let $l$ and $l'$ be the vertical
  line segments in $t$ and $t'$, respectively, joining $p$ to the
  respective hypotenuses ($l'$ could be a single point). If $l$ and
  $l'$ leave $p$ in opposite directions then $l\union l'$ is a
  vertical geodesic connecting two points of $\sigma$, which
  contradicts the uniqueness of geodesics in $\hat X$. If they leave
  $p$ in the same direction but are not equal, then their difference is a vertical
  geodesic with endpoints on $\sigma$, again a contradiction.

We conclude that if  $t$ and $t'$ intersect they do so on a common
subarc of their hypotenuses. This subarc spans a (nonmaximal) right triangle which
is exactly $t\intersect t'$.

Now given $t\in\cT(\sigma)$, the vertical and
horizontal legs of $t$ each contain a single singularity of $\hat X$;
denote these singularities by $v_t$ and $h_t$, respectively. By
construction of $\cT(\sigma)$, there is a unique triangle $t' \in \cT(\sigma)$ such
that $h_{t'} = v_{t}$, unless $v_{t}$ is an endpoint of
$\sigma$. Hence, given an orientation on $\sigma$, the edges of
$\thull(\sigma)$ come with a natural ordering induced by moving along
$\sigma$. By our observations above, we see that $\thull(\sigma)$ is
an embedded arc and meets $\sigma$ only at its endpoints.
Since $\hat X$ is contractible, $\sigma$ and $\thull(\sigma)$ must be
homotopic and hence cobound a disk $P(\sigma)$. In fact this disk is foliated by
both $\lambda^+$ and $\lambda^-$, as we can see by noting that each
edge of $\thull(\sigma)$ cobounds a vertical (similarly a horizontal)
strip with a segment in $\sigma$. Hence $P(\sigma)$ admits an isometry
to a polygon in $\R^2$.
\end{proof}

Let us define a map $\thull^+_\sigma:\sigma\to \thull(\sigma)$ (resp. $\thull^-_\sigma$) which
is the result of pushing the points of $\sigma$ along the
vertical (resp. horizontal) foliation to the other side of $P(\sigma)$. 

If $f\colon I \to \hat X$ is an embedding of an oriented 1-manifold $I$ that parametrizes
some union of saddle connections, we let
\begin{equation}\label{thull f}
  \thull^+ f\colon I \to \hat X
\end{equation}
be the map that sends each $p\in I$ to $\thull^+_\sigma(f(p))$, where
$\sigma$ is the saddle connection containing $f(p)$ with transverse orientation induced by the orientation on $I$. By composing with
covering maps we can use the same notation for the resulting operation
in quotients $\hat X_Y$ or $\bar X$.

Unlike the rectangle hulls, the edges of $\thull(\sigma)$ are not
necessarily $\tau$-edges. (See the upper-right red saddle connection in \Cref{t-example}.)
Moreover, the $\thull$-version of \Cref{r disjoint} is in general not true.
That is, the image of $\thull$ may not project to an embedded complex in $\bar X$ since $\sigma_1$ and $\sigma_2$ can be disjoint while $\thull(\sigma_1)$ and $\thull(\sigma_2)$ cross. However, we do have the following:


\begin{lemma}\label{disjoint_thulls}
Let $\sigma,\sigma'$ be saddle connections in $\hat X$ with disjoint interiors. Let $l$ be an
arc of $\lambda^+$ with endpoints on $\sigma$ and $\sigma'$, and give
$\sigma$ and $\sigma'$ the transverse orientation pointing toward the
interior of $l$. Then the polygons $P(\sigma)$ and $P(\sigma')$ of $\hat X$ (from \Cref{thull structure}) have disjoint interiors.
\end{lemma}

\begin{proof}
Suppose towards a contradiction that there is a point $p$ which is in
the interior of each of the polygons $P = P(\sigma)$ and $P'  =
P(\sigma')$.
Since $P$ and $P'$ are foliated by $\lambda^+$, let $m$ and $m'$ be
the arcs of $\lambda^+$ which are properly embedded in $P$ and $P'$
respectively, and pass through $p$. Orient $m$ so that it begins in
$\sigma$, and $m'$ so that it terminates in $\sigma'$. These
orientations agree at $p$: if they did not we would obtain a
contradiction by applying Gauss--Bonnet to the circuit passing through
$m$,$\sigma$,$l$,$\sigma'$ and $m'$. 

Thus, the union $J=m\union m'$ is an interval in a leaf of $\lambda^+$ with endpoints on
$\sigma$ and $\sigma'$, with $p$ in the interior of $m\intersect m'$. (If $p$ were in $l$ already then
we would have $J=l$.)  
Orienting $J$ as $[y,y']$ where $y\in \sigma$ and $y'\in\sigma'$, we
can write $m=[y,x]$ and $m'=[x',y']$, where $x = J\intersect
\thull(\sigma)$ and $x'= J\intersect \thull(\sigma')$. These points
appear, in order along $J$, as $y,x',p,x,y'$.

\realfig{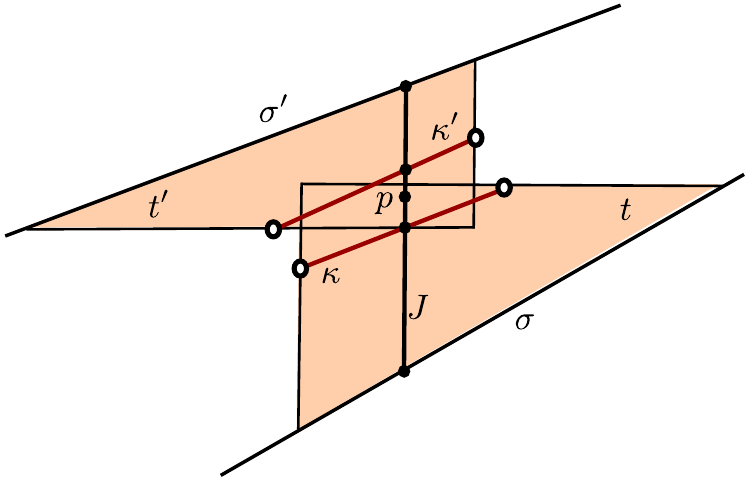}{The point $p$ cannot lie in the interior of both
  $P(\sigma)$ and $P(\sigma')$.}

Let $t$ and $t'$ be the triangles of $\cT(\sigma)$ and $\cT(\sigma')$
containing $x$ and $x'$, respectively. Then $p\in t\intersect t'$.
Let $\kappa$ and $\kappa'$ be the
saddle connections of $\thull(\sigma)$ and
$\thull(\sigma')$ spanning $t$ and $t'$, respectively
(See \Cref{disjoint-P}).
The fact that the endpoints of $\kappa$ and $\kappa'$ are disjoint
from the intersection of $t$ and $t'$ implies that $\kappa\intersect
J$, which is $x$, lies below $\kappa'\intersect J$, which is
$x'$. This contradicts the ordering of the points in $J$. 
\qedhere

%
\end{proof}

\subsection{Retractions in $\A$} \label{sec:totally_geo}
In this subsection, $X$ is fully-punctured.
Let $\A(\tau) \subset \A(X)$ be the span of the vertices of $\A(X)$
which are represented by edges of $\tau$. We will
construct a \emph{coarse 1-Lipschitz retraction} from $\A(X)$ to $\A(\tau)$.
By this, we mean a coarse map which takes diameter $\le 1$ sets to diameter $\le 1$ 
sets and restricts to the identity on the 0-skeleton of $\A(\tau) \subset \A(X)$.

First, let $\mathcal{SC}(q) \subset \A(X)$ be the arcs of $X$ which
can be realized by saddle connections of $q$. Hence,
$\A(\tau) \subset \mathcal{SC}(q) \subset \A(X)$.
For any $a \in \A(X)$
define $\s(a)\subset \mathcal{SC}(q)$ as follows: If $a_q$ is the
$q$-geodesic representative of $a$ in $\bar X$, then let $\s(a)$
be the set of  saddle connections of $q$ composing
$a_q$. If $a$ is a cylinder curve of $q$, then we take $\s(a)$ to be
the set of saddle connections appearing in the boundary of the maximal
cylinder of $a$. Note that if $a \in \A(X)$ is itself represented by a saddle
connection of $q$, then $\s(a)=\{a\}$. 

The following lemma shows that $\s$ is well-defined and is a
coarse $1$-Lipschitz retraction, in the above sense.


\begin{lemma} \label{saddle_proj}
For adjacent vertices $a,b \in \A(X)$,  the vertices of $\s(a)$ and
$\s(b)$ are pairwise adjacent or equal.
\end{lemma} 

\begin{proof}
Recall that adjacency of vertices in $\A(X)$ corresponds to disjointness of their
hyperbolic geodesic representative, and for vertices realized by
saddle connections, this corresponds to the lack of transverse
intersection of their interiors.
But if any arcs of $\s(a)$ and $\s(b)$ have crossing interiors,
 \Cref{cor:side_coherence} implies that the hyperbolic geodesic
representatives of $a$ and $b$ must cross as well. The lemma follows.
\end{proof}

Combining this lemma with \Cref{r disjoint} gives us the proof of
\Cref{th:total geodesic}, which we restate here in somewhat more
precise language:

\restate{th:total geodesic}{
  {\rm (Geodesically connected theorem).}
  Let $(X,q)$ be fully punctured with associated veering triangulation $\tau$. 
The composition $\rhull \circ \s \colon \A(X) \to \A(\tau)$ is a coarse
$1$--Lipschitz retraction in the sense that it takes diameter $\le 1$ sets
to diameter $\le 1$ sets, and is the identity on the $0$-skeleton of $\A(\tau)$.
Hence, any two vertices in $\A(\tau)$ are joined by a geodesic of $\A(X)$ that lies
in $\A(\tau)$.}

\begin{proof}
\Cref{saddle_proj} says that $\s:\A(X)\to\mathcal{SC}(q)$ is a coarse $1$-Lipschitz
retraction. \Cref{r disjoint}, interpreted as a statement about the
arc and curve complexes, says the same for
$\rhull:\mathcal{SC}(q)\to\A(\tau)$. The theorem follows. 
\end{proof}


\section{Projections and compatible subsurfaces}
\label{surface_reps}

In this section we show that if $Y\subset X$ is a compact essential subsurface of large projection distance $d_Y(\lambda^+,\lambda^-)$, then $Y$ has
particularly nice representations with respect to, first, the quadratic
differential $q$ and, second, the veering triangulation $\tau$. We
emphasize that in this section, the surface $X$ is not necessarily
fully-punctured. 


\subsection{Projection and $q$--compatibility}
Recall the $q$-convex hull map $\hat\iota_q \colon Y \to \bar X_Y$
constructed in \Cref{q tight}. We
say that $Y$ is {\em $q$-compatible} 
if $\hat\iota_q$ is an embedding of $Y' = Y\ssm \boundary_0 Y$, as in part (6) of \Cref{q tight}.
(Recall that $\boundary_0 Y$ 
 maps to completion points of $\til \PP_Y$). 
 This condition implies
 a little more:
 
 \realfig{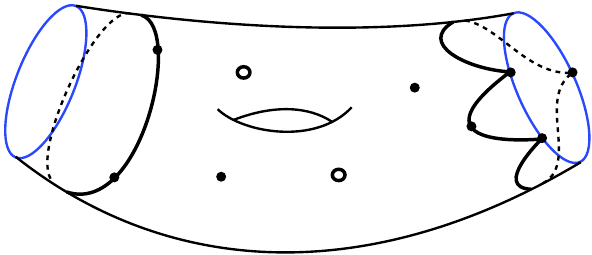}{The image of a $q$-compatible subsurface $Y$ in $\bar X_Y$ under $\hat \iota_q$. Open circles are points of $\til \PP_Y$ (corresponding to the image of $\partial_0 Y$) and dots are singularities not contained in $\til \PP_Y$. The ideal boundary of $X_Y$ is in blue.}


\begin{lemma} \label{lem:int_embed}
  If $Y \subset X$ is $q$-compatible, then
  \begin{enumerate}
    \item the projection $\iota_q \colon Y \to \bar X$ of $\hat
      \iota_q$ to $\bar X$ is an embedding from $\int(Y)$ into $X$
      which is homotopic to the inclusion, and
      \item $\hat\iota_q(\boundary'Y)$ does not pass through points of $\til\PP_Y$.
  \end{enumerate}
\end{lemma}
Recall that $\boundary'Y = \boundary Y \ssm \boundary_0Y$.

\begin{proof}
Recall from \Cref{q tight} that $q$-compatibility of $Y$ is equivalent
to the statement that 
the interior of the $q$-hull $\CH_q(\Lambda) \subset
\hat X$ is a disk (i.e. it is not pinched along singularities or
saddle connections).

If $\iota_q \colon \int(Y) \to X$ fails to be an
embedding, then it must be that for some deck transformation $g$ of
the universal covering $\til X \to X$ the interiors of $\CH_q(\Lambda)$
and $g \cdot \CH_q(\Lambda)$ are distinct and overlap. But then it
follows immediately from \Cref{cor:side_coherence} that the distinct
hyperbolic convex hulls $\CH(\Lambda)$ and $g \cdot \CH(\Lambda)$
overlap, contradicting that $Y$ is a subsurface of $X$. This proves
part (1).

For part (2), let $\beta$ be a component of $\boundary_0 Y$. Since
$\hat\iota_q$ embeds $Y\ssm \boundary_0Y$, a collar
neighborhood $U$ of $\beta$ in $Y$ maps to a neighborhood $V$ of the puncture
$p = \hat\iota_q(\beta)$. Now if $\gamma$ is a component of
$\boundary'Y$, $q$-compatibility again implies its
image must avoid $V\ssm p$. Since $\hat\iota_q(\gamma)$ cannot
equal $p$, it must be disjoint from it. 
\end{proof}

Note that $Y$ is a $q$-compatible annulus if and only if the core of $Y$ is a cylinder curve in $X$. In this case, the corresponding open flat cylinder in $X$ is $\iota_q(\int (Y))$. In general, if $Y$ is $q$-compatible then one component of $X \ssm \partial_q Y$ is an open subsurface isotopic to the interior of $Y$; this is the image $\iota_q(\int(Y))$ and is denoted $\int_q(Y)$.

The following proposition shows that mild assumptions on
$d_Y(\lambda^+,\lambda^-)$ imply that
$Y$ is $q$-compatible.

\begin{proposition}[$q$-Compatibility] \label{prop: q_compatible}
Let $Y\subset X$ be an essential subsurface. 

If $Y$ is non-annular and $d_Y(\lambda^+,\lambda^-) > 0$,  then $Y$ is $q$-compatible. 

If $Y$ is an annulus and $d_Y(\lambda^+,\lambda^-) > 1$, then
$Y$ is $q$-compatible. In this case, $\int_q(Y)$ is a flat cylinder. 
\end{proposition}

\begin{proof}
We treat the non-annular case first. Suppose that $d_Y(\lambda^+,\lambda^-)>0$.
  
Recall from \Cref{AY in flat geometry} that we have identified $\til X$ with $\HH^2$,
set $\Lambda\subset \boundary\HH^2$ to be the limit set of
$\Gamma = \pi_1(Y)$, set $\Omega = 
\boundary\HH^2\ssm \Lambda$, and defined $\hat\PP_Y\subset \Lambda$
to be the set of parabolic fixed points of $\pi_1(Y)$. Note that $\hat \PP_Y = \Lambda \cap \hat \PP$.
Further recall from part (6) of \Cref{q tight} that the
map from $Y'$ to $\CH_q(X_Y)$ is an embedding, 
provided the interior of $\CH_q(\Lambda)$ is a disk.
Since $\CH_q(\Lambda)$ is the result of deleting the interior of the side $\side{l}$ from $\hat X$ for each 
hyperbolic geodesic line $l$ in $\boundary \CH(\Lambda)$, it suffices to show that
\begin{enumerate}
\item for each geodesic line $l$ in $\boundary \CH(\Lambda)$, the
  interior of the corresponding $q$-geodesic $l_q$ does not meet $\partial \HH^2 \ssm \side{l}$, and
\item if $l$ and $l'$ are distinct geodesic lines in $\boundary \CH(\Lambda)$ then
$l_q$ and $l'_q$ do not meet in $\widetilde X$. 
\end{enumerate}

First suppose that condition $(1)$ is violated for some geodesic line $l$ in $\partial \CH(\Lambda)$ and point $\hat p \in \partial \HH^2 \ssm \side{l}$. Set $p$ to be the image of $\hat p$ in $\bar X_Y$. Letting $\gamma$ be the boundary component of $\boundary' Y$ that is the image of $l$ in $X_Y$, we see that the image of $l_q$ in $\bar X_Y$, which equals $\gamma_q = \hat \iota_q(\gamma)$,  passes through the point $p$. 

Since $l_q$ is a geodesic in $\hat X$, we see that $\hat p$ is a completion point and so either $\hat p \in \hat \PP_Y$ or $\hat p \in \hat \PP \ssm \hat \PP_Y$.

Assume that $\hat p\in \hat \PP_Y$. Then
 $p\in \widetilde \PP_Y$ corresponds to a puncture of $Y$.
 Recall that by \Cref{q tight}, the image of the open side
 $\openside{l} = \int(\side{l} \cap \widetilde X)$
 in $X_Y$ is either an open annulus or a disjoint union of open disks; in either case, set $A_\gamma$ equal to the component which contains $p$ in its boundary.
  The angle at $p$ in $A_\gamma$ between the incoming and
  outgoing edges of $\gamma_q$ is at least $\pi$, which implies that
  $A_\gamma$ contains a horizontal and a vertical ray $l^-,l^+$
  emanating from $p$. (\Cref{p_gamma}.)

  \realfig{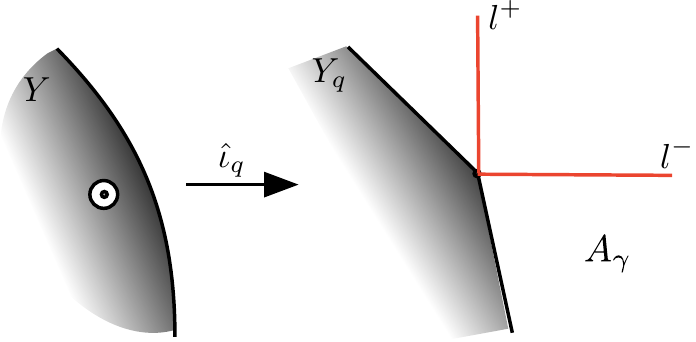}{When $\hat\iota_q(\boundary' Y)$ passes through a 
    point of  $\widetilde \PP_Y$, $d_Y(\lambda^+,\lambda^-) =0$.}
 
  These rays are proper $q$-geodesic lines in $X_Y$ (because $p$ is a
  puncture, not a point of $X_Y$), and hence by \Cref{q arcs for
    AY} represent vertices of $\pi_Y(\lambda^-)$ and
  $\pi_Y(\lambda^+)$, respectively. Further, since the rays only intersect within the annulus or disk $A_\gamma$ and $Y$ is itself nonannular, we see that $l^-$ and $l^+$ in fact represent the same point in $\A(Y)$. (Actually, if $A$ does not contain a flat cylinder, then the interiors of $l^-$ and $l^+$ are disjoint as we show below). Either way, it follows that
  $$
  d_Y(\lambda^+,\lambda^-) =0,
  $$
  a contradiction.
  
 Next assume that $\hat p \in \hat \PP \ssm \hat \PP_Y$. Since $\hat p \notin \side{l} \cap \partial \HH^2$ we may set $A$ to be the component of the image of $\openside{l}$ 
 in $X_Y$ which contains $p \in \hat X_Y$ in its boundary.
 As before, the angle subtended by $\gamma_q$ at $x$ in the boundary of $A$
   is at least $\pi$ (see
  \Cref{modified_pinch-puncture}). A pair of rays $l^\pm$
  emanating from $x$ into $A$ are properly embedded lines and again
  represent the same vertex of $\A(Y)$, giving us
  $d_Y(\lambda^+,\lambda^-) =0$.
  
   \realfig{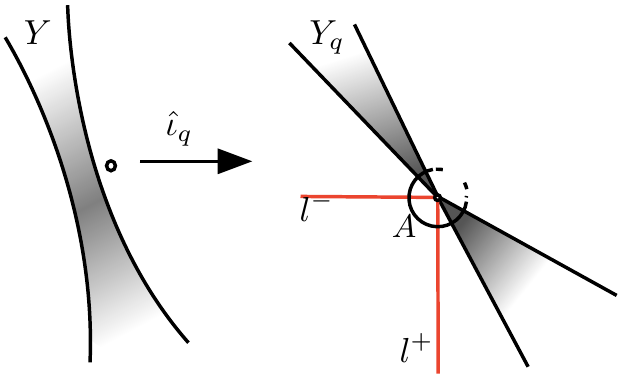}{$Y_q$ is pinched at a completion point.}
  
We conclude that condition $(1)$ is satisfied.
  
Next suppose that geodesics $l$ and $l'$ in the boundary of $\CH(\Lambda)$ violate $(2)$, i.e. $l_q$ and $l'_q$ meet in $\til X$. Let $\til I = l_q \cap l'_q \subset \hat X$ which, since $\hat X$ is CAT$(0)$, is a connected subset of each of $l_q,l'_q$. 
In general, the intersection in $\hat X$ of two $q$-geodesic lines is either a single singularity (possibly a completion point) or a union of saddle connections. 
Because $l_q$ and $l'_q$ meet in $\widetilde X$, $\til I$ contains either a saddle connection or a singularity which is not a completion point.
Let $\gamma,\gamma',\gamma_q,\gamma'_q, I$, be the images in $\bar X_Y$ of $l,l',l_q,l'_q,\til I$, respectively.

Suppose first that $I$ contains a saddle connection $\sigma$.
In this case, let $A$ be the component of the image of 
the open side $\openside{l}$
in $X_Y$ which contains $\sigma$ in its boundary, and define $A'$ similarly. 
(Note that it is possible that $A = A'$ and that $A$ and $A'$ meet along other saddle connections and singularities besides $\sigma$, but this will not change the discussion.)
   
  Any point of $\sigma$ is crossed
  by a pair $l^+,l^-$ of leaves of $\lambda^+,\lambda^-$, which
  as proper arcs of $X_Y$ determine the same vertex of $\A(Y)$. Hence,
  we conclude once again that $d_Y(\lambda^+,\lambda^-) =0.$

  \realfig{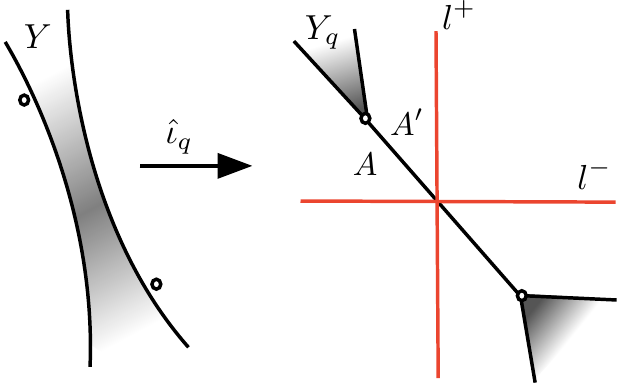}{$Y_q$ is pinched along a saddle connection.}

  Finally, suppose that $I$ contains a singularity $x$ in $X_Y$ (i.e. $x$ is not a completion point). Again, set $A$ to be the component of the image of $\openside{l}$
in $X_Y$ which contains $x$ in its boundary and $A'$ to be the component of the image of $\openside{l'}$ in $X_Y$ which contains $x$ in its boundary.
  As before,   there is an angle of at least $\pi$ on the $A$ side of
  $\gamma_q$ and on the $A'$ side of $\gamma'_q$, 
  so we can find pairs of rays $r_0^\pm$ emanating from $x$ on
  the $A$ side, and $r_1^\pm$ emanating on the $A'$ side (see \Cref{common-sing}).
  The unions
  $l^+ = r_0^+\union_x r_1^+$ and   $l^- = r_0^-\union_x r_1^-$ are
  generalized leaves of $\lambda^+$ and $\lambda^-$, respectively, and 
  again  determine the same point in $\A(Y)$ so we conclude 
  that $d_Y(\lambda^+,\lambda^-) =0. $ 

  \realfig{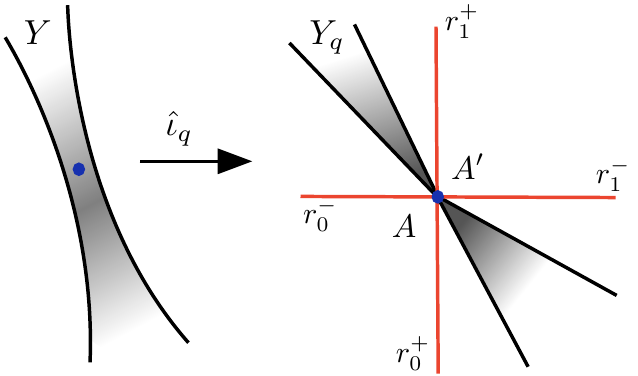}{$Y_q$ is pinched at a singularity which is not a completion point.}

   We conclude that if $Y$ is nonannular and $d_Y(\lambda^+,\lambda^-) >0$, then $Y$ is $q$-compatible.

  When $Y$ is an annulus, almost the same argument applies. The
  difference is that the arcs $l^\pm$ we obtain are not homotopic with
  fixed endpoints, and so do not determine the same vertex of
  $\A(Y)$. However, in each case we will show they have disjoint interiors, 
  concluding $d_Y(l^+,l^-) \le 1$, and so
  $$
  d_Y(\lambda^+,\lambda^-) \le 1.
  $$
  
To see this, let $\gamma$ denote the core of $Y$ and let $\gamma_q$ be a geodesic 
 representative in $\bar X_Y$. Supposing that $\int_q(Y)$ is not a flat annulus, we first claim the following:  For any singular point $p$ crossed by $\gamma_q$, if $l^+$ and $l^-$ are rays of $\lambda^+$ and $\lambda^-$, respectively, meeting with angle $\pi/2$ at $p$,  
then the interiors of $l^+$ and $l^-$ do not meet.

\realfig{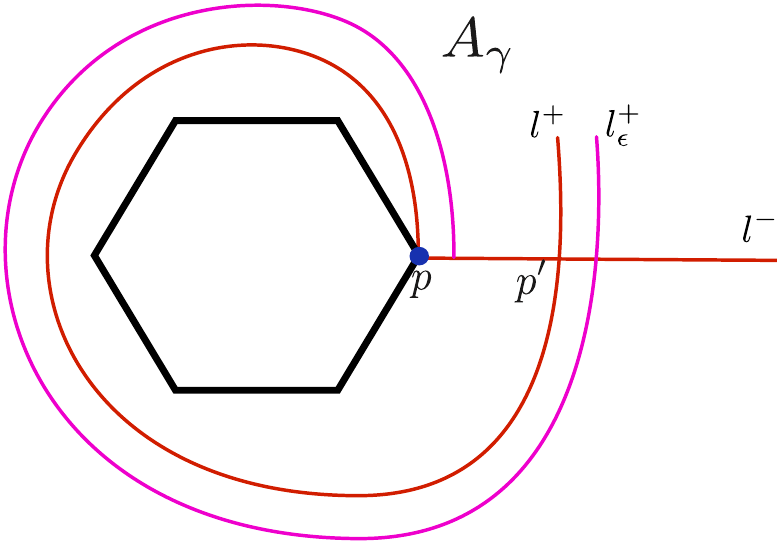}{The $q$-geodesic $\gamma_q$ is the black hexagon. An interior
  intersection between $l^+$ and $l^-$ contradicts the Gauss--Bonnet
  theorem.}
 
 To establish the claim, assume that the interiors of $l^{\pm}$ meet
 and refer to \Cref{GB_annulus_2}. Let $A_\gamma$ be the complementary
 region of $\gamma_q$ in $\bar X_Y$ containing $p'$, the interior intersection of
 $l^{\pm}$. If $A_\gamma$ is a disk, then the claim follows
 immediately from the uniqueness of geodesics in a CAT(0)
 space. Hence, we may assume that $A_\gamma$ is an annulus. Let
 $l^+_\ep$ be leaf of $\lambda^+$ parallel to $l^+$ and slightly
 displaced to the interior of $A_\gamma$, so that the region $R$
 bounded by $\gamma_q$ and the segments of $l^-$ and $l^+_\ep$
 is an annulus.  The total curvature of the $l^-l^+_\ep$ boundary of
 $R$ is 0 since it is straight except for two right turns of opposite
 signs, and the total curvature of $\gamma_q$ as measured from inside
 $R$ is nonpositive (since each singularity on $\gamma_q$ subtends at
 least angle $\pi$ within $R$).  Since $\chi(R)= 0$ and the Gaussian
 curvature in $R$ (including singularities) is nonpositive, the
 Gauss--Bonnet theorem implies that the total curvature of $\boundary
 R$ is nonnegative. This implies that the total curvature of
 $\gamma_q$ is 0, which means that $\gamma_q$ bounds a flat cylinder,
 contradicting our assumption. This establishes the claim.

We now return to the proof of the proposition. First suppose that $\gamma_q$ passes through a completion point $x$ of $\bar X_Y$. Then, just as in \Cref{modified_pinch-puncture}, we can find a pair of rays $l^\pm$ emanating from $x$ into $A_\gamma$. By the claim above, the interiors of these rays do not meet and so $d_Y(\lambda^+,\lambda^-) \le 1$ as desired. 

Finally, suppose that $\gamma_q$ remains in $X_Y$, i.e. it does not
pass through any completion points. It must still pass through a singularity $x$,
and we note that the total angle at $x$ is at
least $3\pi$. Recall that $\gamma_q$ subtends at least angle $\pi$ at
$x$ to either of its sides and we note that some side of $\gamma$ sees
angle at least $3\pi /2$ at $x$. Let $A$ denote this 
side of $\gamma_q$ and let $A'$ denote the other side. 
Note that $A \neq A'$ since $X_Y$ is an annulus which $\gamma_q$ separates.
The angle of $3\pi/2$ tells us there
are at least $3$ rays of $\lambda^\pm$ emanating into $A$. Now choose
rays $r_0^\pm$ of $\l^\pm$ emanating from $x$ on the $A'$ side. 
Because the $3$ (or more) rays of $\lambda^\pm$ emanating from $x$ into
$A$ alternate between $\lambda^+$ and $\lambda^-$, we can choose from
them two rays $r_1^\pm$ of $\l^\pm$ such that
$r_0^+,r_1^+,r_1^-,r_0^-$ are listed 
in the cyclic ordering of directions at $x$ (either
clockwise or counterclockwise).
The generalized leaves $l^+ = r_0^+\union_x r_1^+$
and $l^- = r_0^-\union_x r_1^-$ then represent arcs in the projections
$\pi_Y(\l^+)$ and $\pi_Y(\l^-)$ and after a slight perturbation these
leaves have disjoint interiors. Hence, again we see that
$d_Y(\l^+,\l^-)\le 1$. 

We conclude that  if $Y$ is an annulus with $d_Y(\l^+,\l^-)\ge 1$ then $Y$ is
$q$-compatible. \qedhere

  \end{proof}

\subsection{Projections and $\tau$-compatibility}
\label{sec: hulls_punctured}
We now show how to associate to a subsurface $Y$ of 
large projection a representative of $Y$ which is ``simplicial'' with
respect to the veering
triangulation. This will later be used to prove that such a subsurface induces a 
pocket of the veering triangulation $\tau$.

Informally, we start with a $q$-compatible subsurface $Y \subset X$
and homotope $\hat \iota_q$ by pushing $\partial_q Y$ onto
$\tau$-edges (this process is depicted locally in \Cref{thull-twice}). Formally, this is done in two steps using the map $\thull(\cdot)$
described in \Cref{sec:t_hulls}, although some care must be taken in
order to ensure that the resulting object gives an embedded
representative of $\int(Y)$.

Call a subsurface $Y \subset X$ \emph{$\tau$-compatible} if the map $\hat\iota_q:Y \to \bar X_Y$ is homotopic rel $\partial_0 Y$ to a map
$\hat\iota_\tau:Y\to \bar X_Y$ which is an embedding on $Y' = Y\ssm \partial_0Y$ such that
\begin{enumerate}
\item $\hat\iota_\tau$ takes each component of  $\boundary'Y = \partial Y \ssm \partial_0 Y$ to a simple curve in $\bar X_Y \ssm \til{\PP}_Y$ composed of
a union of $\tau$-edges and 
  \item the map $\iota_\tau \colon Y \to \bar X$ obtained by composing
    $\hat \iota_\tau$ with $\bar X_Y \to \bar X$ restricts to an embedding from $\int(Y)$ into $X$.
\end{enumerate}

We will show that when $d_Y(\lambda^-,\lambda^+)$ is sufficiently large, the subsurface $Y$ is $\tau$-compatible and in this case we set $\partial_\tau Y = \iota_\tau(\partial'Y)$ which is a collection of $\tau$-edges with disjoint interiors. We call $\partial_\tau Y$ the \emph{$\tau$--boundary} of $Y$ and consider it as a $1$-complex of $\tau$-edges. Similar to the situation of a $q$-compatible subsurface, if $Y$ is $\tau$-compatible then one component of $X \ssm \partial_\tau Y$ is an open subsurface isotopic to the interior of $Y$; this is the image $\iota_\tau(\int(Y))$ and is denoted $\int_\tau (Y)$.

\begin{theorem}[$\tau$-Compatibility]\label{thm: tau-compatible}
Let $Y\subset X$ be an essential subsurface. 
\begin{enumerate}
\item If $Y$ is nonannular and $d_Y(\lambda^+,\lambda^-) > 0$, 
 then $Y$ is $\tau$-compatible.
 \item If $Y$ is an annulus and $d_Y(\lambda^+,\lambda^-) > 1$, then
$Y$ is $\tau$-compatible. 
\end{enumerate}
\end{theorem}

\begin{proof}
Suppose that $d_Y(\lambda^+,\lambda^-) > 0$ if $Y$ is nonannular and 
$d_Y(\lambda^+,\lambda^-) > 1$ otherwise.
  By \Cref{prop: q_compatible}, $Y$ is $q$-compatible and so $\hat \iota_q:Y\to \bar X_Y$ is an
  embedding on $Y'$.
  Let $Y_q$ denote its image. We first suppose that $Y$ is not an annulus.

 Give $\partial ' Y$ the transverse orientation pointing into $Y$. For
 any saddle connection $\sigma$ in $\hat\iota_q(\boundary' Y)$ and any
 triangle $t \in \cT(\sigma)$ pointing into $Y$ (see
 \Cref{sec:t_hulls} for definitions), note that the
 singularities of $\bar X_Y$ in $\partial t$ are \emph{not} completion
 points of $\bar X_Y$, that is they do not correspond to punctures of
 $X$. This is because any completion point lying in $t$ is the
 endpoint of leaves $l^\pm$ of $\lambda^\pm$ whose initial segments
 lie in $t$. These leaves correspond to essential proper arcs of $X_Y$
 which are homotopic giving $d_Y(\lambda^-,\lambda^+) =0$, a
 contradiction.
 
 Similarly, we can conclude that for each saddle connection $\sigma$
 in $\hat\iota_q(\boundary' Y)$ and any $t \in \cT(\sigma)$ pointing
 into $Y$, the triangle $t$ is entirely contained in $Y_q$. Otherwise,
 similar to the proof of \Cref{prop: q_compatible}, we find leaves
 $l^+$ and $l^-$ in $\bar X_Y$ whose intersection with $Y_q$ is
 contained in $t$ and hence whose projections to $\A(Y)$ are
 equal. See the left side of \Cref{modified_thull-overlap}. Since $d_Y(\lambda^+,\lambda^-)>0$ this is impossible.
 
  Hence, the map $\thull^+(\hat\iota_q|_{\boundary' Y})$ (as defined in
  (\ref{thull f}) in \Cref{sec:t_hulls}) is
  homotopic to $\hat\iota_q|_{\boundary' Y}$ in $\bar X_Y \ssm \til{\PP}_Y$ 
  by pushing across the polygonal regions
  given by \Cref{thull structure} along leaves of $\lambda^+$. This
  extends to a homotopy of $\hat
  \iota_q$ to a map $\hat \iota' \colon Y \to \bar X_Y$
  which we claim is still an embedding. 
  (Note that, in the case that $X$ is fully-punctured,
  $\hat \iota' = \hat\iota_q$, 
  since all singularities of fully-punctured surfaces are completion points.)

  To prove that $\hat\iota'$ is an embedding, let $C$ be a component of the preimage of
  $\hat\iota_q(Y')$ in 
  $\hat X$
  (using the notation of \Cref{AY in flat geometry}, $C$ is a translate of $\CH_q(\Lambda)$).
   If $\alpha$ is a geodesic segment in $\boundary C$, the
  triangles used in the hull construction are attached to $\alpha$ and are contained in $C$. 
  If such a triangle $t$ intersects a triangle $t'$ from
  a different segment $\alpha'$, they overlap as in the right side of
  \Cref{modified_thull-overlap}. Any  two
  arcs $l^+,l^-$ of $\lambda^+$ and $\lambda^-$ passing through a point
  in the overlap must intersect both $\alpha$ and $\alpha'$.
  These arcs are
  at distance 0 in $\A(Y)$, since they can be isotoped to each other rel $\boundary
  Y$.
  Hence $d_Y(\lambda^-,\lambda^+) = 0$, contradicting the
  hypothesis. Therefore, $t,t'$ cannot overlap.


\realfig{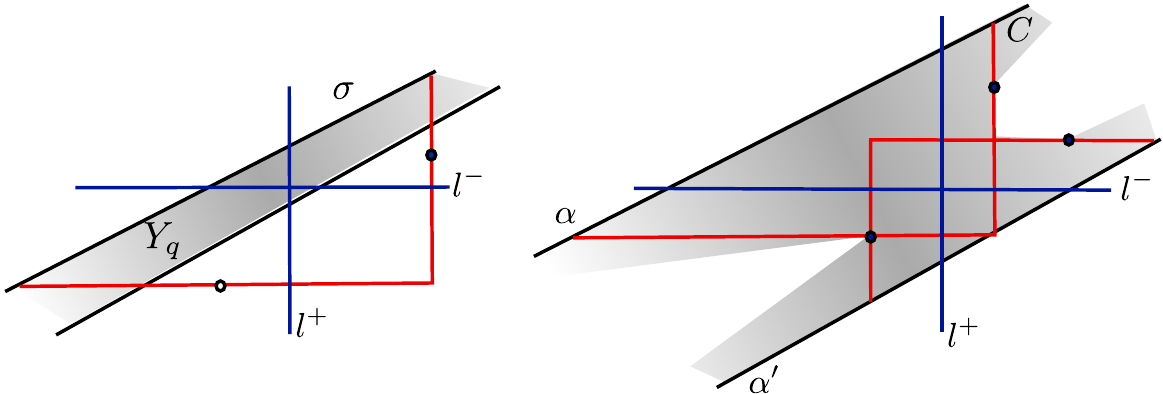}{Left: If $t \in \cT(\sigma)$ (in red) is not contained in $Y_q$ then $d_Y(\lambda^+,\lambda^-) =0$. Right:
An overlap of two hull triangles. Any completion point in the boundary of a hull triangle does not correspond to a puncture in $\til{\PP}_Y$.}

We conclude that the polygonal regions of our homotopy are embedded
and disjoint, and thus the homotopy can be chosen so that
$\hat \iota'$ is an
embedding. Since the image of $\hat \iota'$ is contained in the image of $\hat \iota_q$, we apply \Cref{lem:int_embed} to get that the projection $\iota' \colon Y \to \bar X$ restricts to an embedding on $\int(Y)$.


  Now orient $\boundary'Y$ in the opposite direction, pointing out of
  the surface, and apply $\thull$ again, this time to $\hat \iota'(\partial' Y)$.
  The triangles in the construction now extend outside the surface,
  and the result of the operation is the rectangle hull
  $\rhull(\thull(\hat\iota_q(\boundary' Y)))$,
 which is therefore composed of $\tau$-edges. 
  Using the homotopy pushing $\hat \iota'|_{\partial' Y}$
  outward along leaves of $\lambda^+$ to $\thull^+(\hat\iota'|_{\boundary' Y})$ (again using
   \Cref{thull structure}) we obtain our final map 
  $\hat\iota_\tau$. See \Cref{thull-twice}. It remains to show that 
  $\hat\iota_\tau \colon Y  \to \bar X_Y$ has the required properties.
To prove this, let us recapitulate the construction in the universal
cover. 


  \realfig{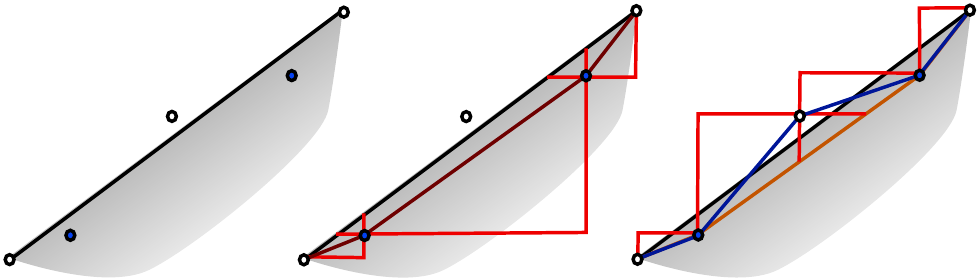}{An inner $\thull$ followed by outer $\thull$
    yields $\tau$-edges. This locally depicts the homotopy from $\hat \iota_q$ to $\hat \iota_\tau$.}

As before, let $C = \CH_q(\Lambda)$. The map $\hat \iota'$ lifts to a
$\pi_1(Y)$-equivariant homeomorphism $C \to C'$, where $C'$ is
obtained by giving each saddle connection $\kappa$ in the boundary of
$C$ the transverse orientation pointing into $C$ and removing the
polygons $P(\kappa)$ given in \Cref{thull structure}. This map is
equivariantly homotopic to the identity by pushing along leaves of the
vertical foliation. 

The outward step of our construction then pushes back along leaves of
the vertical foliation to obtain a $\pi_1(Y)$-equivariant map $C'\to
C_\tau\subset \hat X$, so that the composition $C\to C'\to C_\tau$ is a lift of the map
$\hat\iota_\tau \colon Y'
\to \bar X_Y$. 
To show that $\hat\iota_\tau \colon Y'  \to
\bar X_Y$ is an embedding, it suffices to show that the composition
$C\to C_\tau$ is a homeomorphism.


For every non-singular point $p\in \boundary C$ there is an arc $n_p$ in
$\lambda^+$ such that the deformation of $C$ to $C_\tau$ is supported
on the union $\bigcup n_p$, and preserves each $n_p$. Thus to show
that $C\to C_\tau$ is a homeomorphism it suffices to show that $n_p
\intersect n_{p'} =\emptyset$ for each $p\ne p'$ in $\boundary C$.
The interior pieces,  $n_p\intersect C$, are already disjoint for
distinct points, by our
construction. Thus if $n_p$ intersects $n_{p'}$ their union is an
interval $J$ in a leaf of $\lambda^+$ with some subinterval between $p$ and $p'$ lying outside
$C$. This contradicts the convexity of $C$.


To show that $\iota_\tau$ is an embedding when restricted to $\int(Y)$, it suffices to check that the interior of $C_\tau$ is disjoint from all its translations under the entire deck group $\pi_1(X)$. To see this, take $g \in \pi_1(X)$ so that $C_\tau$ and $g \cdot C_\tau$ are distinct and intersect. Since $\iota' \colon \int(Y) \to X$ is an embedding, $C'$ and $g \cdot C'$ meet only along their boundary. Further, if $\sigma$ is a saddle connection in $\partial C' \cap \partial (g \cdot C')$, then $\sigma$ is the hypotenuse of a singularity-free triangle pointing into $C'$ as well as one pointing into $g\cdot C'$. Hence, $\sigma$ is a $\tau$-edge and so is fixed under the map $C' \to C_\tau$. 

Now if the interiors of $C_\tau$ and $g\cdot C_\tau$ intersect there must be saddle connections $\sigma,\kappa \subset \partial C'$ 
such that
$P(\sigma)$ and  $P(g \cdot \kappa)$ have intersecting interiors. (Here, $\sigma,\kappa$ are oriented out of $C'$.) 
By the previous paragraph, $\sigma$ and $g\cdot \kappa$ are distinct.
As $C'$ and $g \cdot C'$ meet only along their boundary, $\sigma$ and
$g \cdot \kappa$ have disjoint interiors and any arc $l$ of $\lambda^+$ joining $\sigma$ to $g \cdot \kappa$ within $P(\sigma) \cup P(g \cdot \kappa)$ lives outside of $C'$ and $g \cdot C'$. In particular, the chosen transverse orientations on $\sigma$ and $g \cdot \kappa$ point to the interior of $l$.
However, by \Cref{disjoint_thulls}, in this situation, the interiors
of $P(\sigma)$ and  $P(g \cdot \kappa)$ do not intersect. It follows that $\iota_\tau \colon \int(Y) \to  X$ is an embedding.

It only remains to prove property $(1)$ of the definition of $\tau$-compatible. 
Since $\hat \iota_\tau \colon Y' \to
\bar X_Y$ is an embedding, it follows that 
$\hat \iota_\tau|_{\partial' Y}$ is an embedding, and its image
does not meet
$\widetilde P_Y$ by the same argument used to prove item $(2)$ of
\Cref{lem:int_embed}. 
By construction the image $\hat \iota_\tau (\partial' Y)$ is composed of
$\tau$-edges.

Now suppose that $Y$ is an annulus. Then $\hat \iota_q(Y)$ is the
(nondegenerate) maximal flat cylinder of $\bar X_Y$ by \Cref{prop:
  q_compatible}. Choosing the inward-pointing orientation for
$\boundary Y$, we claim that 
$\thull^+(\hat \iota_q|_{\partial Y}) = \hat \iota_q|_{\partial Y}$:
Otherwise, there must be a
saddle connection $\sigma$ on the boundary of the flat annulus
$\hat \iota_q(Y)$, and a triangle $t$ pointing into the annulus
with hypotenuse on $\sigma$, which encounters a singularity or
puncture $x$ on the other side of the annulus.
The picture is similar to the left side of \Cref{modified_thull-overlap}.
A variation on the Gauss--Bonnet argument
in the annulus case of \Cref{prop: q_compatible} then produces 
vertical and horizontal leaves passing through $x$ which have disjoint
representatives, and hence $d_Y(\lambda^+,\lambda^-) \le 1$.
Thus the inward step of the process is the identity, and the outward
step and the rest of the proof proceed just as in the nonannular case. 
\end{proof}

\begin{remark} \label{rmk:fully}
From the proof of \Cref{thm: tau-compatible}, we record the fact that if $X$ is fully-punctured and $Y$ satisfies 
the hypotheses of \Cref{thm: tau-compatible}, then $\thull^+(\hat \iota_q|_{\partial Y}) = \hat \iota_q|_{\partial Y}$
and $\iota' = \hat \iota_q$.
Hence, in this case we have that $\partial_\tau Y = \rhull(\partial_q Y)$.
\end{remark}

\section{Embedded pockets of the veering triangulation and bounded projections}
\label{pockets}

In this section, let $X$ be fully-punctured with respect to the
foliations $\lambda^\pm$ of a pseudo-Anosov $f:X\to X$, and let $M$ be
the mapping torus. Recall that every fiber associated to the fibered
face $\FF$ of $X$ must also be fully-punctured because they are transverse
to the same suspension flow, and hence that $\FF$ is
a \emph{fully-punctured fibered face}. 

We now prove our two main theorems on the structure of subsurface projections 
in a fully-punctured fibered face, \Cref{th:bounding projections} and
\Cref{th:sub_dichotomy_fully_punctured}. The main tools in the proof
are the structure and embedding theorems for pockets associated with
high-distance subsurfaces, which we develop below. Recall that $\diam_Z(\cdot)$ denotes the diameter of $\pi_Z(\cdot)$ in $\A(Z)$ and that subsurfaces $Y$ and $Z$ \emph{overlap} if, up to isotopy, they are neither disjoint nor nested.

\subsection{Projections and $\tau$--compatible subsurfaces}
We begin by discussing projection to $\tau$-compatible subsurfaces.

\begin{lemma} \label{lem:overlap_tau}
Let $Y$ and $Z$ be $\tau$-compatible subsurfaces of $X$ and let $K \subset X$ 
be a disjoint collection of saddle connections which correspond to edges from $\tau$. 
Then 
\begin{enumerate}
\item If $K$ meets  $\int_\tau(Y)$, then $\pi_Y(K) \neq \emptyset$, and  $\diam_Y(\pi_Y (K)) \le 1$. 
\item If $Y$ and $Z$ are disjoint, then so are $\int_\tau(Y)$ and $\int_\tau(Z)$.
\item If $Y$ and $Z$ overlap, then $\diam_Z(\partial Y \union \partial_\tau Y) \le 1$.
\item The subsurface $\int_\tau(Y)$ is in minimal position with the foliations $\lambda^\pm$. In particular, the arcs of $\int_\tau(Y) \cap \lambda^\pm$ agree with the arcs of $\pi_Y(\lambda^\pm)$.
\end{enumerate}
\end{lemma}

\begin{proof}
For item (1), the main point is to show that an edge of $K$ that
meets $\int_\tau(Y)$ lifts to an {\em essential edge} in $\bar X_Y$.
This is true for edges meeting $\int_q(Y)$, using the local CAT(0) geometry
of $\bar X_Y$ and the fact that $\hat\iota_q(Y')$ is a locally convex embedding.


Thus it will suffice to show that any $\tau$-edge $e$ meeting
$\int_\tau(Y)$ must also meet $\int_q(Y)$.
Suppose, on the contrary, that $e$ meets $\int_\tau(Y)$ but not $\int_q(Y)$.
Then $e$ meets the interior of a polygon $P(\sigma)$ where $\sigma$ is
an outward-oriented saddle connection in $\boundary_q Y$
(recall from \Cref{rmk:fully} 
that, since $X$ is fully-punctured, the
inner $\thull$ step in the construction of $\hat\iota_\tau$ is the
identity, and the outer $\thull$ is in fact a rectangle hull).
Let
$R$ be the singularity-free rectangle spanned by $e$. If $e$ is
contained in $P(\sigma)$ then $R$ can be extended to a rectangle whose
diagonal lies in $\sigma$, and hence $e$ is one of the edges of
$\rhull(\sigma)$; but this contradicts the assumption that $e$ meets
$\int_\tau(Y)$. Thus $e$
crosses some edge $f$ of
$\rhull(\sigma)$. However, $f$ is contained in a singularity-free
triangle whose hypotenuse lies along $\sigma$ and so $\sigma$ must
cross the rectangle $R$ either top-to-bottom or side-to-side. In
either case, we see that $e$ crosses $\sigma \subset \partial_q Y$,
a contradiction.
We conclude that if a $\tau$-edge meets $\int_\tau(Y)$, then it also meets $\int_q(Y)$ 
and hence has a well-defined projection to $Y$.
The diameter bound in item $(1)$ is then immediate since $K$ is a disjoint collection of essential arcs of $\A(X)$. 

For item $(2)$, first note that when $Y$ and $Z$ are disjoint subsurfaces of $X$, 
the interiors $\int_q(Y)$ and $\int_q(Z)$ are also disjoint. This follows from \Cref{cor:side_coherence} and the $q$-hulls construction in \Cref{q tight}. More precisely, 
let $\Lambda_Y$ and $\Lambda_Z$ be the limit sets of $Y$ and $Z$ in $\partial \HH^2$ (using our identifications from \Cref{AY in flat geometry}). Since $Y$ and $Z$ do not intersect, $\Lambda_Y$ and $\Lambda_Z$ do not link in $\partial \HH^2$ and so $\CH_q(\Lambda_Y)$ and $\CH_q(\Lambda_Z)$ have disjoint interiors by \Cref{cor:side_coherence}. This implies that $\int_q(Y)$ and $\int_q(Z)$ are disjoint in $X$.

To obtain $\int_\tau (Y)$ from $\int_q (Y)$ we append to each saddle
connection $\sigma$ in $\partial_q Y$ the (open) polygon $P(\sigma)$,
where $\sigma$ is oriented out of $Y$. We obtain $\int_\tau (Z)$  from
$\int_q (Z)$ by the same construction. 
Since  $\int_q(Y)$ and $\int_q(Z)$ are disjoint in $X$, it suffices to show that $P(\sigma)$ and $P(\kappa)$ have disjoint interiors, where $\sigma \subset \partial_qY$ and $\kappa \subset \partial_q Z$. If $\sigma = \kappa$, then this saddle connection spans a singularity-free rectangle and $P(\sigma) = \sigma = \kappa = P(\kappa)$. Otherwise, $\sigma$ and $\kappa$ have disjoint interiors and \Cref{disjoint_thulls} implies that $P(\sigma)$ and $P(\kappa)$ have disjoint interiors, as required.
This proves item (2).


Since $\int_\tau(Y)$ is an embedded representative of the interior of $Y$, $\partial Y$
has a representative disjoint from the collection of saddle
connections in $\partial_\tau Y$. Hence $\diam_Z(\partial Y \union
\partial_\tau Y) \le 1$, proving item $(3)$. For
item $(4)$, first note that the subsurface $\int_q(Y)$ is in minimal position with the foliations $\lambda^\pm$. This is immediate from the local CAT$(0)$ geometry in $\bar X_Y$ and
the fact that $\lambda^\pm$ are geodesic: any bigon in $\bar X_Y$ between $\hat \iota_q(\partial' Y)$ and a leaf of $\lambda^\pm$ would lift to a bigon in $\hat X$ bounded by two geodesic segments, a contradiction to uniqueness of geodesics in $\hat X$.
 The statement for
$\int_\tau(Y)$ then follows from the fact that the homotopy from
$\boundary_q Y$ to $\boundary_\tau Y$ can be taken to move either along
vertical or along horizontal leaves, using either $\thull^+$ or $\thull^-$ as in 
the proof of \Cref{thm: tau-compatible}. 
\end{proof}

\subsection{Pockets for a $\tau$-compatible subsurface} 
Suppose that $Y \subset X$ is \linebreak $\tau$--compatible. By
\Cref{cor:top}, the set $T(\boundary_\tau Y)$ of sections containing
$\boundary_\tau Y$ contains a top and a bottom section, denoted $T^+ =
T^+(\boundary_\tau Y)$
and $T^- = T^-(\boundary_\tau Y)$, which between them bound a number of
pockets. See \Cref{sections} for terminology related to sections and pockets.
Our assumption on $d_Y(\lambda^-,\lambda^+)$ will imply that one
of these pockets is isotopic to a thickening of $Y$, as explained in
the following proposition:

\begin{proposition}[Pockets in $\tau$]\label{Y pocket}
  Let $(X,q)$ be fully-punctured and $Y\subset X$ an essential
  nonannular subsurface. 
  \begin{enumerate}
    \item If $d_Y(\lambda^-,\lambda^+) > 0$ then
$d_Y(T^+,\lambda^+) = d_Y(T^-,\lambda^-) = 0$. 
    \item If $d_Y(\lambda^-,\lambda^+) > 2$ then
 $T^+$ and  $T^-$ bound
      a pocket $U_Y$ whose interior is isotopic to a thickening of
      $\int(Y)$.
  \end{enumerate}
  When $Y$ is an annulus, 
  \begin{enumerate}
    \item If $d_Y(\lambda^-,\lambda^+) > 1$ then
$d_Y(T^+,\lambda^+) = d_Y(T^-,\lambda^-)  = 1$. 
    \item If $d_Y(\lambda^-,\lambda^+) > 4$ then
 $T^+$ and  $T^-$ bound
      a pocket $U_Y$ whose interior is isotopic to a thickening of
      $\int(Y)$.
  \end{enumerate}

\end{proposition}

\begin{proof}
Begin with the following lemma:

\begin{lemma}\label{near lambda}
Suppose that $Y \subset X$
is $\tau$-compatible,
let $e$ be an edge of $\boundary_\tau Y$ and let $f$ be a $\tau$-edge
crossing $e$ with $f>e$. Then $d_Y(f,\lambda^+) \le 1$ if $Y$ is an annulus and $d_Y(f,\lambda^+) =0$ otherwise. Similarly if
$f<e$ then the same statement holds for $d_Y(f,\lambda^-)$.
\end{lemma}

\realfig{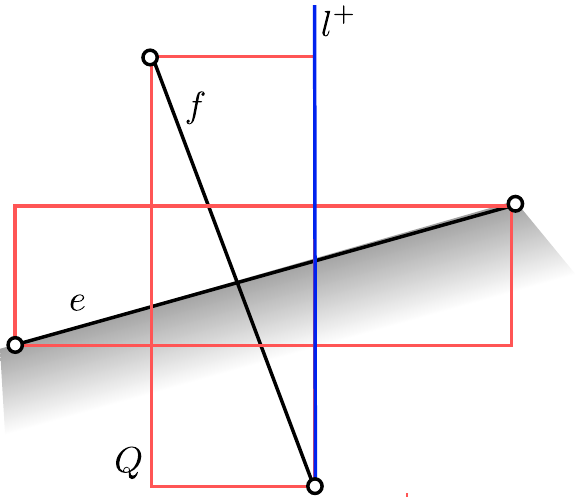}{Local picture near the $\tau$-edge $e$ of  $\partial_\tau Y$ with $\int_\tau(Y) \subset X$ shaded. When $f>e$, the edge $l^+$ of $Q$ represents
  $\pi_Y(\lambda^+)$ and is disjoint from $f$. Note that $Q$ is \emph{immersed} in $X$.}

The key idea of the proof is pictured in \Cref{f-above-e_2}. Here it is shown that if $f$ crosses $e \subset \partial_\tau Y$ with $f>e$, then some component of the intersection of $f$ with $\int_\tau(Y)$ is disjoint from some arc in $\pi_Y(\lambda^+)$. However, the spanning rectangle $Q$ for $f$ is immersed in $X$ (rather than necessarily embedded).
To handle this issue, we work in the cover $\widetilde X$.

\begin{proof}
Let $C^{\mathrm{o}}$ be a component of the preimage of $\int_\tau(Y)$ under $\widetilde X \to X$ and choose a saddle connection $\til e$ in the boundary of $C^{\mathrm{o}}$ which projects to $e$. 
Further, let $\til f$ be any lift of $f$ which crosses $\til e$.
Since $f$ is a $\tau$-edge, $\til f$ spans a singularity-free rectangle $\til Q$ whose immersed image in $X$ we denote by $Q$. 

Every $\tau$-edge which crosses $\til Q$ does so either
top to bottom or side to side.
Since $f>e$, $\til e$ must cross $\til Q$ from side to side (see
\Cref{sections}). 
Since all $\tau$-edges in $\boundary C^{\mathrm{o}}$ are disjoint, they all must
cross $\til Q$ from side to side. 

Since $\int_\tau(Y)$ is in minimal position with $\lambda^+$
(\Cref{lem:overlap_tau}), $C^{\mathrm{o}}$ intersects each leaf of the vertical 
foliation in a connected set. 
Together these observations
imply that $\til Q \intersect C^{\mathrm{o}}$ is a single polygon $\til B$,
bounded by at least one edge crossing $\til Q$ from side to side 
(which we have called $\til e)$. See \Cref{f-above-e_cover}.



\begin{figure}[htbp]
\begin{center}
\includegraphics[scale = .7]{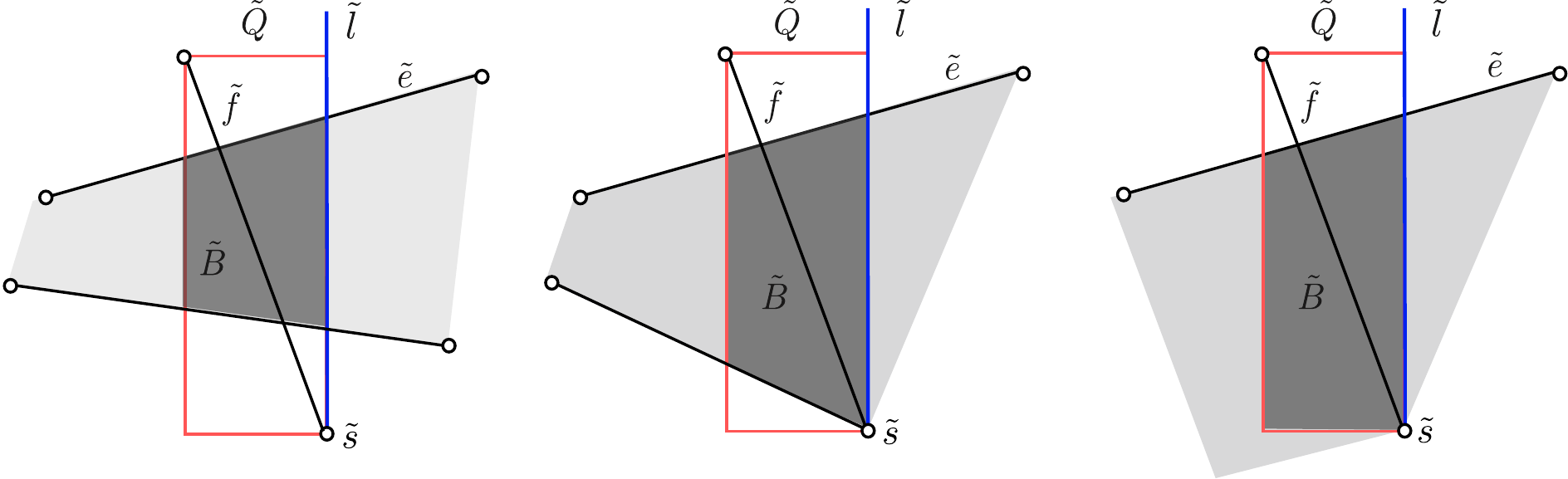}
\caption{The 3 possibilities for $\til B$. The lightly shaded region is part of $C^{\mathrm{o}}$ in $\widetilde X$.}
\label{f-above-e_cover}
\end{center}
\end{figure}

\begin{claim*}
$\til B$ embeds in $\int_\tau (Y)$ under the covering $\til X \to X$.
\end{claim*}

\begin{proof}[Proof of claim]
Since $\til B \subset C^{\mathrm{o}}$, the image of $\til B$ is contained in $\int_\tau (Y)$.
Suppose that $x,y \in \til B$ map to the same point in $\int_\tau(Y)$, and denote by $l_x$ and $l_y$ the vertical leaf segments in $\til X$ starting at $x$ and $y$, respectively, and continuing to $\til e$. Since $\til B$ is convex, 
$l_x,l_y \subset \til B$. Suppose that $l_x$ is no longer than $l_y$ and let $l_y'$ be the subsegment of $l_y$ with length equal to that of $l_x$. Then $l_x$ and $l_y'$ are identified under the map $\til X \to X$. But the identification of $ \partial l_x \ssm \{x\} \subset \til e$ and $ \partial l'_y \ssm \{y\} \subset \til B \cup \til e \subset C^{\mathrm{o}} \cup \til e$ gives a contradiction, unless $x=y$: the edge $\til e$ is mapped injectively into $X$ with image $e\subset \partial_\tau Y$ disjoint from the image of $C^{\mathrm{o}}$, which is $\int_\tau (Y)$.
\end{proof}

Let $\til s$ be the vertex of $\til f$ which
is on the same side of $\til e$ as $\til B$.
Let $\til l$ be the vertical side of $\til Q$ starting at $\til s$. 

Let $B$ be the image of $\til B$ in  $X$. By the claim, $B$ is a
singularity-free quadrilateral in $X$ whose interior is contained in
$\int_\tau(Y)$. The images in $X$ of $\til f \intersect \til B$ and $\til l
\intersect \til B$ are therefore disjoint proper arcs in
$\int_\tau(Y)$, which by 
\Cref{lem:overlap_tau} are representatives of
$\pi_Y(f)$ and $\pi_Y(\lambda^+)$, respectively. Moreover, these arcs are properly 
homotopic in $\int_\tau(Y)$ by a homotopy supported in $B$.

Hence, when $Y$ is nonannular, we conclude that  $d_Y(f,\lambda^+) =0$.  
If $Y$ is an annulus, we project the picture to the annular cover
$X_Y$, where
 we note that the image $l$ of $\til l$, continued to infinity, cannot intersect
 $f$ without meeting $Q$, and hence $e$, again. Since $l$ can only
 meet $\boundary_q Y$ once in the annular cover, we conclude it is
 disjoint from $f$ and so
$d_Y(f,\lambda^+) =1$.

 The case $f<e$ is similar, so \Cref{near lambda} is proved.
\end{proof}

%
%

We return to the proof of \Cref{Y pocket}.
Let $Y$ be nonannular.
Note that by definition the only
upward-flippable edges in $T^+$ must lie in $\boundary_\tau Y$. Let
$e$ be such an edge and consider the single flip move that replaces
$e$ with an edge $f$. Then $f>e$, so by \Cref{near lambda},
$d_Y(f,\lambda^+) = 0$. On the other hand $f$ and $e$ are diagonals
of a quadrilateral made of edges of $T^+$, at least one of which, $e'$,
gives the same element of $\A(Y)$ as $f$. Hence 
$d_Y(T^+,\lambda^+) = 0.$

If $Y$ is an annulus, we note that $e'$ and the vertical leaf in the
proof of \Cref{near lambda} give adjacent vertices of $\A(Y)$, so 
$d_Y(T^+,\lambda^+)\le 1$. Note that $d_Y(T^+,\lambda^+) \ne 0$ because no 
leaf of the foliation $\lambda^+$ has both its endpoints terminating at
completion points.

To prove the statements about pockets, let $K$ be the common edges of  
$T^+$ and  $T^-$,
viewed as a subcomplex of $X$.
If $\int_\tau(Y)$ contains an edge of $K$ then from the triangle inequality,
 together with the first part of the proposition, 
we obtain $d_Y(\lambda^+,\lambda^-) \le 2$ when $Y$ is
nonannular, and 
$d_Y(\lambda^+,\lambda^-) \le 4$ when $Y$ is an annulus. By our
hypotheses this does not happen, so we conclude that 
$T^+,T^- \in T(\partial_\tau Y)$ have no common edges contained
in $\int_\tau (Y)$. Hence $T^+$ and $T^-$ bound a pocket $U_Y$ whose base is
$\int_\tau(Y)$. This completes the proof.
\end{proof}

\subsection{Isolated pockets and projection bounds}

Let $X$ be a fiber in $\R_+\FF$, and 
let $Y$ be a $\tau$--compatible subsurface of $X$ 
such that $d_Y(\lambda^-,\lambda^+)>4$.
An \emph{isolated pocket} for $Y$ in $(X \times \mathbb{R}, \tau)$ is a subpocket $V = V_Y$ of $U_Y$ with base $\int_\tau (Y)$ such that 
\begin{enumerate}
\item For each edge $e$ of $V$ which is not contained in  $\partial_\tau Y$,
\[
d_Y(e,\lambda^+) \ge 3 \quad \text{and} \quad d_Y(e,\lambda^-) \ge 3
\]
if $Y$ is nonannular, and 
\[
d_Y(e,\lambda^+) \ge 4 \quad \text{and} \quad d_Y(e,\lambda^-) \ge 4
\]
if $Y$ is an annulus.
\item Denoting by $V^\pm$ the top and bottom of $V$ with their induced triangulations, 
\[
d_Y(V^-, V^+) \ge 1.
\]
\end{enumerate}

Note that condition $(2)$ guarantees that $\mathrm{int}(V_Y) \cong \int_\tau(Y) \times (0,1)$ is still a pocket just as in \Cref{Y pocket}. The next lemma shows that for $Y$ with $d_Y(\lambda^-,\lambda^+)$ sufficiently large, $Y$ has an isolated pocket with $d_Y(V^-,V^+)$ roughly $d_Y(\lambda^-,\lambda^+)$.

\begin{lemma} \label{lem:iso_pocket}
Suppose that $Y$ is a nonannular subsurface of $X$ with
$d_Y(\lambda^-,\lambda^+) > 8$. Then $Y$ has an isolated pocket $V$
with $d_Y(V^-,V^+) \ge d_Y(\lambda^-,\lambda^+) - 8$.

If $Y$ is an annulus with
$d_Y(\lambda^-,\lambda^+) > 10$, then $Y$ has an isolated pocket $V$
with $d_Y(V^-,V^+) \ge d_Y(\lambda^-,\lambda^+) - 10$.
\end{lemma}

\begin{proof}
  Let $c=4$ if $Y$ is an annulus and $c=3$ otherwise, and assume that
  $d_Y(\lambda^+,\lambda^-) > 2c+2$.
Since the pocket $U = U_Y$ is connected (\Cref{prop:connect}), there is a sequence of
sections  $T^- = T_0,  T_1, \ldots, T_N = T^+$ in $T(\partial_\tau Y)$
such that $T_{i+1}$ differs from $T_i$
by an upward diagonal exchange. From \Cref{Y pocket}, we know that
$d_Y(T^- , \lambda^-) \le 1$ and $d_Y(T^+ , \lambda^+) \le 1$. Let $0
< a < N$ be largest integer such that  
$d_Y(T_{a-1}, \lambda^-) < c$; hence $d_Y(T_{i}, \lambda^-) \ge c$
for all $i\ge a$. Now let $b<N$ be the smallest integer greater than
$a$ such that
$d_Y(T_{b+1},\lambda^+) < c$; then $d_Y(T_{i},\lambda^+) \ge c$ for
all $a\le i \le b$.

Note that these indices exist since $d_Y(\lambda^-,\lambda^+) \ge 2c+1$.

Now let $V$ be the pocket between $T_a$ and $T_b$ with base contained in
$\int_\tau(Y)$ and note that $V$ is a subpocket of $U$. 
Any edge $e$ of $V$ not contained in $\boundary_\tau Y$ is
contained in a section
$T_i \in T(\partial_\tau Y)$ for $a \le i \le b$. Since we have
$d_Y(T_i,\lambda^\pm)\ge c$, we have $d_Y(e,\lambda^\pm)\ge c$. Thus it
only remains to get a lower bound on $d_Y(V^+,V^-)$.

The triangle
inequality (and diameter bound on $T_a$ and $T_b$) gives
\[
d_Y (V^-,V^+) = d_Y(T_a,T_b) \ge d_Y(\lambda^-,\lambda^+) - 2c -2 \ge 1.
\]
This implies that $\int_\tau (Y)$ is the base of $V$ and completes the proof. 
\end{proof}

The following proposition shows that isolated pockets coming from either disjoint or overlapping subsurfaces of $X$ have interiors which do not meet.

\begin{proposition}[Disjoint pockets] \label{prop:disjoint_pockets}
Suppose that $Y$ and $Z$ are subsurfaces of $X$ with isolated pockets $V_Y$ and $V_Z$. Then, up to switching $Y$ and $Z$, either $Y$ is nested in $Z$, or the isolated pockets $V_Y$ and $V_Z$ have disjoint interiors in $X \times \mathbb R$.
\end{proposition}

\begin{proof} 
If the subsurfaces $Y$ and $Z$ are disjoint, then $\int_\tau(Y)$ and $\int_\tau(Z)$ are also disjoint by \Cref{lem:overlap_tau}. Hence, the maximal pockets $U_Y$ and $U_Z$ have disjoint interiors by definition. 

Now suppose that $Y$ is not an annulus.
We claim that if $Y$ and $Z$ overlap then either
\[
d_Y(\partial_\tau Z, \lambda^+)\le 1 \; \text{ or } \; d_Y(\partial_\tau Z,\lambda^-) \le 1.
\]
To see this, first note that there is some edge $f$ 
contained in $\int_\tau(Z)$
such that $f$ crosses some edges of $\partial_\tau Y$. 
Otherwise, every triangulation of $\int_\tau(Z)$ by $\tau$--edges
would contain edges from $\partial_\tau Y$. But then applying this to
$T^\pm(\partial_\tau Z)$ and using \Cref{Y pocket}, 
we would have that
\[
d_Z(\lambda^-,\lambda^+) \le 2 + \diam_Z(\partial_\tau Y) \le 3, 
\]
contradicting our assumption on the subsurface $Z$.
Now if $f$ intersects an edge $e$ of $\partial_\tau Y$ and $f>e$, then
by \Cref{near lambda}, $d_Y(\partial_\tau Z, \lambda^+)\le
d_Y(\partial_\tau Z, f) \le 1$. If $f<e$ then \Cref{near lambda} gives
$d_Y(\partial_\tau Z, \lambda^-)\le d_Y(\partial_\tau Z, f) \le 1$. 

Now suppose that $e$ is an edge of $U_Y \cap U_Z$ which is not
contained in $\partial_\tau Y \cup \partial_\tau Z$. Then $e$, as a
$\tau$-edge in $X$, is disjoint from $\partial_\tau Z$ and so
$d_Y(e,\lambda^+) \le 2$ or
$d_Y(e,\lambda^-) \le 2$.
Hence $e$ cannot be contained in $V_Y$.
We conclude that $V_Y \cap V_Z \subset \partial_\tau Y \cup
\partial_\tau Z$. This completes the proof when $Y$ is not an
annulus. 

When $Y$ is an annulus, then a similar argument using the annular case of \Cref{near lambda} shows that if $Y$ and $Z$ overlap then either
\[
d_Y(\partial_\tau Z, \lambda^+)\le 2 \; \text{ or } \; d_Y(\partial_\tau Z,\lambda^-) \le 2.
\]
Hence, if $e$ is an edge of $U_Y \cap U_Z$ which is not
contained in $\partial_\tau Y \cup \partial_\tau Z$, then
$d_Y(e,\lambda^\pm) \le 3$.  So again $e$ cannot be contained in $V_Y$ and we conclude that $V_Y \cap V_Z \subset \partial_\tau Y \cup \partial_\tau Z$ as required.
\end{proof}

We next prove that isolated pockets embed into the fibered manifold
$M$. This is \Cref{thm:pocket summary}, which we restate here in more
precise language. 

\restate{thm:pocket summary}{
{\rm (Embedding the pocket).}
  Suppose $Y$ is a subsurface of a fully-punctured fiber $X$ with
  $d_Y(\lambda^-,\lambda^+) > \beta$, where $\beta=8$ if $Y$ is
  nonannular and $\beta=10$ if $Y$ is an annulus.
Then $Y$ has an isolated pocket $V_Y$ in $X \times \R$, and the
covering map $X \times \R \to M$ restricts to an embedding of the
subcomplex $V_Y \to M$. 
}

\begin{proof}
Let $\Phi$ be the simplicial isomorphism of $X \times \R$ induced by $f$ as in \Cref{gueritaud construction}.
Note that if $T$ is a section of $\tau$, then $\Phi (T)$ is the section of $\tau$ whose corresponding triangulation of $X$ is $f(T)$. 
Hence, $\Phi(T(\partial_\tau Y)) = T(\partial_\tau{f (Y)})$.

By \Cref{lem:iso_pocket}, $Y$ has an 
isolated pocket $V =V_Y$. Note that $V$ embeds into $M$ if and only if it is disjoint from its translates $V_i = \Phi^i(V)$ for each $i \neq 0$. By the remark above, each $V_i$ is itself an isolated pocket for the subsurface $Y_i = f^i(Y)$, and any two of these subsurfaces are either disjoint or overlap in $X$. Hence, by \Cref{prop:disjoint_pockets} the isolated pockets $V_i$ are disjoint as required. 
\end{proof}

We will now prove \Cref{th:bounding projections}, whose
statement we recall here:

\restate{th:bounding projections}{
Let $M$ be a hyperbolic 3-manifold with fully-punctured fibered face
$\FF$ and veering triangulation $\tau$.
For any subsurface $W$ of any fiber of $\FF$,
\[
\alpha \cdot (d_W(\lambda^- ,\lambda^+) -\beta) < |\tau|, 
\]
where $|\tau|$ is the number of tetrahedra in $\tau$, 
$\alpha = 1$ and $\beta = 10$ when $W$ is an annulus and
$\alpha = 3|\chi(W)|$ and $\beta = 8$ when $W$ is not an annulus.
}

\begin{proof}
Suppose that $W$ is any nonannular subsurface of any fiber $F$ in $\R_+  \FF$.
We may assume that $d_W(\lambda^-,\lambda^+) >8$. 
Then \Cref{lem:iso_pocket} implies that $W$ has an isolated pocket $V_W$ in $(F \times \R, \tau)$ such that $d_W(V_W^-,V_W^+)\ge d_Y(\lambda^-,\lambda^+) -8$.
By \Cref{thm:pocket summary}, the isolated pocket $V_W \subset (F \times
\R, \tau)$ embeds into $(M,\tau)$. Hence
$|V_W| \le |\tau|$, where $|V_W|$ denotes the number of
tetrahedra of $V_W$. 
Now each tetrahedron of $V_W$ corresponds to a diagonal exchange between
the triangulations $V_W^-$ and $V_W^+$ of $W_\tau$ and each diagonal
exchange replaces a single edge of the triangulation. 
There are at least $3|\chi(W)| + 1$ non-boundary edges to each triangulation of $W$,
and the diameter in $\A(W)$ of an ideal triangulation is 1, so we conclude
\begin{align} \label{ineq:pocket_growth}
|\tau| &\ge |V_W| = \#\{\text{diagonal exchanges from } V_W^- \text{ to } V_W^+\}\\
&> 3|\chi(W)| \cdot d_W(V^-,V^+) \nonumber \\
&\ge 3|\chi(W)| \cdot (d_W(\lambda^-,\lambda^+) - 8) \nonumber.
\end{align}
This completes the proof when $W$ is nonannular. 

When $W$ is an annulus, we use the annular case of
\Cref{lem:iso_pocket} to obtain an isolated pocket $V_W$ in $(F \times
\R, \tau)$ such that $d_W(V_W^-,V_W^+)\ge d_Y(\lambda^-,\lambda^+)
-10$. Noting that a triangulation of the annulus contains at least 2
(non-boundary) edges, the same argument implies that  
\begin{align*}
|\tau| &\ge |V_W| = \#\{\text{diagonal exchanges from } V_W^- \text{ to } V_W^+\}\\
&> d_W(V^-,V^+) \nonumber \\
&\ge d_W(\lambda^-,\lambda^+) - 10 \nonumber, 
\end{align*}
as required.
\end{proof}

\subsection{Sweeping through embedded pockets}
We are now ready to prove \Cref{th:sub_dichotomy_fully_punctured},
whose statement we reproduce below. 
This theorem relates subsurfaces of large projections among different fibers of a
fixed face. 

\restate{th:sub_dichotomy_fully_punctured}{
Let $M$ be a hyperbolic 3-manifold with fully-punctured fibered face
$\FF$ and suppose that $S$ and $F$ are each fibers in $\R_+\FF$.
If $W$ is a subsurface of $F$, then either $W$ is isotopic
along the flow to a subsurface of $S$, or
$$3|\chi(S)| \ge d_W(\lambda^-,\lambda^+) -\beta,$$
where $\beta =10$ if $W$ is an annulus and $\beta = 8$ otherwise.
}

Recall from \Cref{lem:subgroup_projection} that
we can identify $d_W(\lambda^+,\lambda^-)$ with
$d_W(\Lambda^+,\Lambda^-)$, agreeing with the statement given in the introduction.

We will require the following lemma, which essentially states that immersed subsurfaces with large projection are necessarily covers of subsurfaces. Recall that in \Cref{sec: arc_complex} we defined the distance $d_W(\lambda^+,\lambda^-)$ when $W$ is a compact core of a cover $X_\Gamma \to X$ corresponding to a finitely generated subgroup $\Gamma \le \pi_1(X)$.

\begin{lemma}[Immersion to cover] \label{lem:embedding_fullly_punctured}
Suppose that $(X,q)$ is a fully-punctured surface.
Let $\Gamma$ be a finitely generated subgroup of $\pi_1(X)$ and let $W$ be a compact core of the cover $X_\Gamma \to X$. 
If $W$ is nonannular and $d_W(\lambda^-,\lambda^+) > 4$ or 
if $W$ is an annulus and
$d_W(\lambda^-,\lambda^+) >6$,
then there is a subsurface $Y$ of $X$ such that 
$W \to X$ is homotopic to a finite cover $W \to Y \subset X$. 

In particular, $\Gamma$ is
a finite index subgroup of $\pi_1(Y)$.
\end{lemma}

\begin{proof}
Suppose that $d_W(\lambda^-,\lambda^+) > 4$ if $W$ is nonannular and 
$d_W(\lambda^-,\lambda^+) >6$ if $W$ is an annulus.
Let $p \colon \check X \to X$ be a finite cover to which $W \to X$
lifts to an embedding $W \to \check{X}$ (this exists since surface
groups are LERF \cite{scott-LERF}),
and identify $W$ with its
image in $\check{X}$. Lift $q$ along with the veering triangulation to
$(\check{X} \times \R,\tau)$. By \Cref{thm: tau-compatible},
$W$ is a $\tau$--compatible subsurface of $\check X$, and
by \Cref{thm: tau-compatible}
 and  \Cref{prop:connect}, 
$T_{\check{X}}(\partial_\tau W)$ is nonempty and connected. To prove the lemma, we
 show that $\int_\tau(W) \to X$ covers a subsurface of $X$. For this, it suffices to prove 
 that each edge of $p^{-1}(p(\partial_\tau W))$ is disjoint from $\int_\tau(W)$. Indeed, since
 $W$ is $\tau$--compatible, one component of $\check X \ssm \partial_\tau W$ is $\int_\tau(W)$. If $p^{-1}(p(\partial_\tau W))$ is disjoint from $\int_\tau(W)$, then $\int_\tau(W)$ is also a component of $\check X \ssm p^{-1}(p(\partial_\tau W))$. As components of $\check X \ssm p^{-1}(p(\partial_\tau W))$ cover components of $X \ssm p(\partial_\tau W)$, this will show that $\int_\tau(W) \to X$ covers a subsurface of $X$.

%

Hence, we must show that each edge of $p^{-1}(p(\partial_\tau W))$ is disjoint from $\int_\tau(W)$. This is equivalent to the statement that no edge of $p^{-1}(p(\partial_\tau W))$ crosses $\partial_\tau W$ nor is contained in $\int_\tau(W)$.

First suppose that $W$ is not an annulus.
If $\check T$ is a section
of $(\check{X} \times \R,\tau)$ with an edge $f$ such that $f>e$ for
an edge $e$ of $\partial_\tau W$,
then \Cref{near lambda} implies that
$d_W(\check T,\lambda^+) = 0$. Similarly if $f<e$ then  $d_W(\check T,\lambda^-) =
0$. Hence, if $T$ is \emph{any section of} $(X \times \R,\tau)$ such that
$d_W(T,\lambda^\pm) \ge 1$, then its lift
$\check{T} = p^{-1}(T)$ to $\check{X}$ must contain the edges of
$\partial_\tau W$ and so $\check{T} \in T_{\check{X}}(\partial_\tau
W)$.  Moreover, such a section $T$ of $(X \times \R,\tau)$ with $d_W(T,\lambda^\pm) \ge 1$ must exist. 
This is because
by \Cref{gue-sweep}, we may sweep through $X\times\R$ with sections going from
near $\lambda^-$ to near $\lambda^+$. If all sections were to have $d_W$--distance
$0$ from either $\lambda^-$ or $\lambda^+$, then there would be a pair $T,T'$ differing by
a single diagonal exchange such that $d_W(T, \lambda^-)= d_W(T',\lambda^+) = 0$. But this would imply that 
$d_W(\lambda^-,\lambda^+) \le 2$, contradicting our assumption on distance.

Putting these facts together, we conclude that 
 there exists a section $T$ of $(X \times \R,\tau)$ with $d_W(T,\lambda^\pm)\ge 1$, and that
for each such section 
\[
 p^{-1}(T) \in T_{\check{X}}(p^{-1}(p(\partial_\tau W))).
 \]
Note that this in particular implies that no edge of $p^{-1}(p(\partial_\tau W))$
crosses an edge of $\partial_\tau W$.
 

We claim now that no edge
$e$ in $p^{-1}(p(\partial_\tau W))$ can be
contained in $\int_\tau (W)$. Such an edge would have
a well-defined projection to
$\A(W)$  and 
would necessarily appear in each section of $T_{\check{X}}(p^{-1}(p(\partial_\tau W)))$ (by definition of $T_{\check X}(\cdot)$).
Using our conclusion from above, this would imply that
$d_W(p^{-1}(T),e) = 0$ whenever $d_W(T,\lambda^\pm)\ge 1$. 
But just as before, by sweeping through $X\times\R$ with sections going from
near $\lambda^-$ to near $\lambda^+$, 
we produce sections $T_1, T_2$ with $d_W(T_1, \lambda^-) =d_W(T_2,\lambda^+)=1$.
Since each of these sections' preimage in $\check X$ contains the edge $e$, we get that
$d_W(\lambda^\pm,e)\le 2$, which contradicts our hypothesis that
$d_W(\lambda^+,\lambda^-) > 4$.

This shows that no edge of $p^{-1}(p(\partial_\tau W))$ can meet $\int_\tau(W)$ and 
completes the proof when $W$ is nonannular.
When $W$ is an annulus, one proceeds exactly as above using the annular version of \Cref{near lambda}.
\end{proof}

\begin{proof}[Proof of \Cref{th:sub_dichotomy_fully_punctured}]
  We may assume that $W$ is a subsurface of $F$ such that
  $d_W(\lambda^-,\lambda^+) > \beta$. 

First suppose that $\pi_1(W)$ is contained in $\pi_1(S)$.
Then by \Cref{lem:embedding_fullly_punctured}, there is a subsurface
$Y$ of $S$ such that, up to conjugation in $\pi_1(S)$, $\pi_1(W) \le
\pi_1(Y)$ is a finite index subgroup; let
$n \ge 1$ denote this index. If $\eta_F \colon \pi_1(M) \to
\Z$ is the homomorphism representing the cohomology class dual to
$F$, then $\eta_F | \pi_1(Y)$ vanishes on the index $n$ subgroup
$\pi_1(W)$.  
Since $\Z$ is torsion-free we must have that $\eta$ vanishes on
$\pi_1(Y)$ and hence $\pi_1(Y)$ is contained in $\pi_1(F)$.
However, since the fundamental group of an embedded subsurface, in this case $W \subset F$, can not be nontrivially finite-index inside another subgroup of $\pi_1(F)$, we see that $n=1$ and $\pi_1(W) = \pi_1(Y)$.
That $W$ is isotopic along the flow in $M$ to $Y \subset S$ can be seen by lifting $W$ and $Y$ to the cover $S \times \mathbb{R} \to M$.


Hence, we may suppose by \Cref{lem:flow_to_fiber_2} that the image of
any $S \to M$ homotopic to the fiber $S$ intersects any isotope of  $W
\subset F$ essentially. Since $d_W(\lambda^-,\lambda^+) > \beta$, $W$
has a nonempty isolated pocket $V_W \subset F \times \mathbb R$
which simplicially embeds into $(M, \tau)$ by
\Cref{thm:pocket summary}. Let $\{W_i\}$ denote a sequence of sections of $V_W$ from $V^-_W$ to $V^+_W$ with $W_{i+1}$ differing from $W_i$ by an upward diagonal flip. Also, fix a simplicial map $f \colon S \to (M,\tau)$ which is obtained by composing a section of $(S \times \mathbb{R},\tau)$ with the covering map $S \times \mathbb{R} \to M$. 

Note that for each $i$, $f(S)$ meets at least one edge of the interior of $W_i$. Otherwise, the image of $S$ in $M$ misses the interior of $W_i$ contradicting our assumption. In fact, even more is true: Call a component $c$ of $f(S) \cap W_i$ \emph{ removable} if the triangles of $f(S)$
incident to the edges of $c$
lie locally to one side of $W_i$ in $M$. If $c$ is removable, then
there is an isotopy of $W_i$ supported in a neighborhood of $c$ which
removes $c$ from the intersection $f(S) \cap W_i$. Hence, if we denote
by $E_i$ the edges of $f(S) \cap W_i$ which do not lie in removable components , then $E_i$ must be nonempty for each $i$. 

We claim that for each $i$, $E_i$ shares an edge with $E_{i+1}$. Otherwise, both $E_i$ and $E_{i+1}$ consist of a single edge and
the tetrahedron corresponding to the diagonal exchange from $W_i$ to
$W_{i+1}$ has $E_i$ as its bottom edge and $E_{i+1}$ as its top
edge. But then both of these edges must be removable since pushing the bottom two faces of the tetrahedron slightly upward makes that intersection disappear, and similarly for the top. This contradicts our above observation and establishes that $E_i$ and $E_{i+1}$ have a common edge. 

We obtain a sequence in $\A(W)$,
\[
V^-_W \supset E_0 , E_1, \ldots, E_n \subset V^+_W,
\]
having the property that for each edge $e_i$ of $E_i$ there is an edge $e_{i+1}$ of $E_{i+1}$ such that $e_i$ and $e_{i+1}$ are disjoint. We conclude that the number of distinct edges in the sequence $E_0 , E_1, \ldots, E_n$ is at least $d_W(V^-_W, V^+_W)$.
Combining this with the fact that the number of edges in an ideal triangulation
of $S$ is $3|\chi(S)|$ and \Cref{lem:iso_pocket}, we see that
\[
3|\chi(S)| \ge d_W(V^-_W, V^+_W)  \ge d_W(\lambda^-,\lambda^+) - \beta,
\]
as required. 
\end{proof}

\medskip

We conclude the paper by recording the following corollary of
\Cref{lem:embedding_fullly_punctured} and the proof of
\Cref{th:sub_dichotomy_fully_punctured}. 

\begin{corollary}\label{always subsurface}
Let $M$ be a hyperbolic manifold with fully-punctured fibered face
$\FF$. Let $W$ be a subsurface of a fiber $F\in\R_+\FF$ such that
$d_W(\Lambda^+,\Lambda^-) > 4$ if $W$ is nonannular and
$d_W(\Lambda^+,\Lambda^-) > 6$ if $W$ is an annulus. If $S$ is any
fiber in $\R_+\FF$ such that $\pi_1(W) < \pi_1(S)$, then $W$ is isotopic
to a subsurface of $S$.
\end{corollary}

\bibliography{pfnew.bbl}

\def\cprime{$'$} \def\cprime{$'$}
\providecommand{\bysame}{\leavevmode\hbox to3em{\hrulefill}\thinspace}
\providecommand{\MR}{\relax\ifhmode\unskip\space\fi MR }
\providecommand{\MRhref}[2]{%
  \href{http://www.ams.org/mathscinet-getitem?mr=#1}{#2}
}
\providecommand{\href}[2]{#2}
\begin{thebibliography}{BKMM12}

\bibitem[Ago11]{agol2011ideal}
Ian Agol, \emph{Ideal triangulations of pseudo-{A}nosov mapping tori}, Topology
  and geometry in dimension three \textbf{560} (2011), 1--17.

\bibitem[Ago12]{agol-overflow}
\bysame, \emph{Comparing layered triangulations of 3-manifolds which fiber over
  the circle}, MathOverflow discussion,
  {http://mathoverflow.net/questions/106426}, 2012.

\bibitem[BCM12]{ELC2}
J.~Brock, R.~Canary, and Y.~Minsky, \emph{The classification of {K}leinian
  surface groups, {II}: {T}he ending lamination conjecture}, Ann. of Math.
  \textbf{176} (2012), no.~1, 1--149. \MR{2925381}

\bibitem[BKMM12]{BKMM}
Jason Behrstock, Bruce Kleiner, Yair Minsky, and Lee Mosher, \emph{Geometry and
  rigidity of mapping class groups}, Geom. Topol. \textbf{16} (2012), 781--888.

\bibitem[BS05]{BSc}
David Bachman and Saul Schleimer, \emph{{Surface bundles versus {H}eegaard
  splittings}}, Communications in Analysis and Geometry (2005), no.~5, 903 --
  928.

\bibitem[CC00]{candel2000foliations}
Alberto Candel and Lawrence Conlon, \emph{Foliations {II}}, American
  Mathematical Society Providence, 2000.

\bibitem[FG13]{futer2013explicit}
David Futer and Fran{\c{c}}ois Gu{\'e}ritaud, \emph{Explicit angle structures
  for veering triangulations}, Algebraic \& Geometric Topology \textbf{13}
  (2013), no.~1, 205--235.

\bibitem[FLM11]{farb-leininger-margalit}
B.~Farb, C.~J. Leininger, and D.~Margalit, \emph{Small dilatation
  pseudo-{A}nosov homeomorphisms and 3-manifolds}, Adv. Math. \textbf{228}
  (2011), no.~3, 1466--1502. \MR{2824561}

\bibitem[Fri82]{fried1982geometry}
David Fried, \emph{The geometry of cross sections to flows}, Topology
  \textbf{21} (1982), no.~4, 353--371.

\bibitem[Gu{\'e}15]{gueritaud}
Fran{\c{c}}ois Gu{\'e}ritaud, \emph{Veering triangulations and the
  {C}annon-{T}hurston map}, ArXiv:1506.03387, 2015.

\bibitem[Har02]{hartshorn}
K.~Hartshorn, \emph{Heegaard splittings of {H}aken manifolds have bounded
  distance}, Pacific J. Math. \textbf{204} (2002), no.~1, 61--75. \MR{MR1905192
  (2003a:57037)}

\bibitem[HIS16]{hodgson2016non}
Craig~D Hodgson, Ahmad Issa, and Henry Segerman, \emph{Non-geometric veering
  triangulations}, Experimental Mathematics \textbf{25} (2016), no.~1, 17--45.

\bibitem[HRST11]{hodgson2011veering}
Craig~D Hodgson, J~Hyam Rubinstein, Henry Segerman, and Stephan Tillmann,
  \emph{Veering triangulations admit strict angle structures}, Geometry \&
  Topology \textbf{15} (2011), no.~4, 2073--2089.

\bibitem[JMM10]{johnson-minsky-moriah:subsurface}
J.~Johnson, Y.~Minsky, and Y.~Moriah, \emph{Heegaard splittings with large
  subsurface distances}, Algebr. Geom. Topol. \textbf{10} (2010), no.~4,
  2251--2275. \MR{2745671 (2012g:57038)}

\bibitem[McM00]{mcmullen2000polynomial}
Curtis~T McMullen, \emph{Polynomial invariants for fibered 3-manifolds and
  {T}eichm{\"u}ller geodesics for foliations}, Annales scientifiques de l'Ecole
  normale sup{\'e}rieure \textbf{33} (2000), no.~4, 519--560.

\bibitem[Min10]{ECL1}
Yair Minsky, \emph{The classification of {K}leinian surface groups, {I}:
  {M}odels and bounds}, Ann. of Math. (2010), 1--107.

\bibitem[MM00]{MM2}
Howard~A. Masur and Yair~N. Minsky, \emph{Geometry of the complex of curves.
  {II}. {H}ierarchical structure}, Geom. Funct. Anal. \textbf{10} (2000),
  no.~4, 902--974.

\bibitem[MS13]{masur2013geometry}
Howard Masur and Saul Schleimer, \emph{The geometry of the disk complex},
  Journal of the American Mathematical Society \textbf{26} (2013), no.~1,
  1--62.

\bibitem[Raf05]{rafi2005characterization}
Kasra Rafi, \emph{A characterization of short curves of a {T}eichm{\"u}ller
  geodesic}, Geometry \& Topology \textbf{9} (2005), no.~1, 179--202.

\bibitem[Sco78]{scott-LERF}
Peter Scott, \emph{Subgroups of surface groups are almost geometric}, J. London
  Math. Soc. (2) \textbf{17} (1978), no.~3, 555--565. \MR{0494062}

\bibitem[ST06]{scharlemann-tomova}
Martin Scharlemann and Maggy Tomova, \emph{Alternate {H}eegaard genus bounds
  distance}, Geom. Topol. \textbf{10} (2006), 593--617 (electronic).
  \MR{2224466}

\bibitem[Thu86]{thurston1986norm}
William~P Thurston, \emph{A norm for the homology of 3-manifolds}, Mem. Amer.
  Math. Soc. \textbf{59} (1986), no.~339, 99--130.

\bibitem[TW15]{tang-webb}
Robert Tang and Richard Webb, \emph{{Shadows of Teichm\"uller disks in the
  curve graph}}, arXiv:1510.04259, 2015.

\end{thebibliography}
\bibliographystyle{amsalpha}

\end{document}